\newcommand{\pointsize}{11pt}
\documentclass[oneside, \pointsize]{amsbook}
\usepackage[
   includehead,
   includefoot,
     left = 1in, 
      top = 1in, 
    right = 1in,
   bottom = 1in
]{geometry}
\usepackage{fancyhdr}
\usepackage{setspace}
\usepackage{calc}
\usepackage[nocompress]{cite} 
\usepackage[pdfborder={0 0 0}, pdfpagemode=UseNone, pdfstartview=FitH]{hyperref} 

\fancyheadoffset[R]{0.5in} 

\fancypagestyle{prelim}{%
   \renewcommand{\headrulewidth}{0pt} 
   \fancyhf{}           
   \pagenumbering{roman}    
   \cfoot{-\thepage-}       
}

\fancypagestyle{maintext}{%
   \renewcommand{\headrulewidth}{0.4pt}
   \pagenumbering{arabic}
   \fancyhf{}
   \fancyhead[L]{\rightmark}
   \cfoot{\thepage}
}

\fancypagestyle{abstract1}{%
   \renewcommand{\headrulewidth}{0.4pt}
	 \fancyheadoffset[R]{0in}
   \pagenumbering{arabic}
   \fancyhf{}
}

\fancypagestyle{abstract2}{%
   \renewcommand{\headrulewidth}{0.4pt}
   \fancyheadoffset[R]{0in}
	 \pagenumbering{arabic}
   \fancyhf{}
   \rhead{-\thepage-}
}

\numberwithin{figure}{chapter} 
\numberwithin{table}{chapter}
\numberwithin{equation}{chapter}
\numberwithin{section}{chapter}

\usepackage{graphicx}
\usepackage{caption} 
\usepackage{subcaption} 
\usepackage{amsmath}
\usepackage{tkz-berge}
\usepackage{tkz-graph}
\usepackage{pgf}
\usetikzlibrary{snakes, arrows, shapes, positioning, graphs, petri, topaths, chains, matrix, scopes}
\usepackage{amsthm}
\usepackage{chngcntr}
\usepackage{tikz}
\usepackage{mathrsfs}
\usepackage[utf8]{inputenc}

\newtheorem{theorem}{Theorem}[section]
\newtheorem{proposition}{Proposition}[section]
\newtheorem{conjecture}{Conjecture}[section]

\newtheorem{lemma}[theorem]{Lemma}
\newtheorem*{remark*}{Remark}
\newtheorem*{theorem*}{Theorem}
\newtheorem{question}{Question}[section]
\newtheorem*{question*}{Question}
\newtheorem{definition}{Definition}[section]
\newtheorem*{definition*}{Definition}
\newtheorem{innercustomthm}{Theorem}
\newenvironment{customthm}[1]
  {\renewcommand\theinnercustomthm{#1}\innercustomthm}
  {\endinnercustomthm}

\begin{document}

   \frontmatter

   \pagestyle{prelim}
   
   %
   \fancypagestyle{plain}{%
      \fancyhf{}
      \cfoot{-\thepage-}
   }%
   \begin{center}
   \null\vfill
   \textbf{%
      Nearly Finitary Matroids
   }%
   \\
   \bigskip
   By \\
   \bigskip
   Patrick C. Tam \\
   DISSERTATION \\
   \bigskip
   Submitted in partial satisfaction of the requirements for the
   degree of \\
   \bigskip
   DOCTOR OF PHILOSOPHY \\
   \bigskip
   in \\
   \bigskip
   MATHEMATICS \\
   \bigskip
   in the \\
   \bigskip
   OFFICE OF GRADUATE STUDIES \\
   \bigskip        
   of the \\
   \bigskip
   UNIVERSITY OF CALIFORNIA \\
   \bigskip
   DAVIS \\
   \bigskip
   Approved: \\
   \bigskip
   \bigskip
   \makebox[3in]{\hrulefill} \\
   Eric Babson \\
   \bigskip
   \bigskip
   \makebox[3in]{\hrulefill} \\
   Jesús De Loera\\
   \bigskip
   \bigskip
   \makebox[3in]{\hrulefill} \\
   Matthias Köppe\\
   \bigskip
   Committee in Charge \\
   \bigskip
   2018 \\
   \vfill
\end{center}

   \newpage
	
	 \thispagestyle{empty}
	 \begin{titlepage}
	 \vspace*{50em}
	 \begin{center}
		 \copyright \ Patrick C. Tam, 2018.  All rights reserved.  
	 \end{center}
	 \end{titlepage}
	 \newpage
	 \stepcounter{page}
	
	 \vspace*{20em}
	 \begin{center}
	   To...
	 \end{center}
	 \newpage
   
   %
   \doublespacing
   
   \tableofcontents
   \newpage
   
   {\singlespacing
   \begin{flushright}
      Patrick C. Tam \\
      March 2018 \\
      Mathematics \\
   \end{flushright}
}

\bigskip

\begin{center}
   Nearly Finitary Matroids \\
\end{center}

\section*{Abstract}

In this thesis, we study nearly finitary matroids by introducing new definitions and prove various properties of nearly finitary matroids.  In 2010, an axiom system for infinite matroids was proposed by Bruhn et al. We use this axiom system for this thesis. In Chapter 2, we summarize our main results after reviewing historical background and motivation. In Chapter 3, we define a notion of spectrum for matroids. Moreover, we show that the spectrum of a nearly finitary matroid can be larger than any fixed finite size. We also give an example of a matroid with infinitely large spectrum that is not nearly finitary. Assuming the existence of a single matroid that is nearly finitary but not $k$-nearly finitary, we construct classes of matroids that are nearly finitary but not $k$-nearly finitary. We also show that finite rank matroids are unionable. In Chapter 4, we will introduce a notion of near finitarization. We also give an example of a nearly finitary independence system that is not $k$-nearly finitary. This independence system is not a matroid. In Chapter 5, we will talk about Psi-matroids and introduce a possible generalization. Moreover, we study these new matroids to search for an example of a nearly finitary matroid that is not $k$-nearly finitary. We have not yet found such an example. In Chapter 6, we will discuss thin sums matroids and consider our problem restricted to this class of matroids.

Our results are motivated by the open problem concerning whether every nearly finitary matroid is $k$-nearly finitary for some $k$.
   \newpage
   
   \section*{Acknowledgments}
   I would like to thank my adviser, Eric Babson, for his patience and continual guidance. I particularly appreciate his kindness and great composure. I would also like to thank my committee for their helpful feedback on this dissertation. I would also like to thank Nathan Bowler and Hiu Fai Law for useful discussions. I also thank the Infinite Matroid community for developing the theory that makes my research possible.  I would also like to thank my friends Jessica Shum, Michael Glaros, Eric Samperton, Cecilia Dao, John Murray, Alexander Berrian, Alvin Moon, Ryan Halabi, Christopher David Westenberger, Evan Smothers and Colin Hagemeyer for making my time at Davis more enjoyable. Lastly, I want to thank both of my parents for their continual support of me throughout my entire life.
   
   \mainmatter
   
   \pagestyle{maintext}
   
   %
   \fancypagestyle{plain}{%
      \renewcommand{\headrulewidth}{0pt}
      \fancyhf{}
      \cfoot{\thepage}
   }%
   
   \chapter{Introduction}
   \label{ch:IntroductionLabel}
   Matroids generalize the notion of linear independence in linear algebra. Matroid axioms were first developed for finite matroids. Finite matroids have been extensively studied since 1935. Many examples of finite matroids come from finite graphs and these have been extensively studied. It is natural to ask whether there is a good notion of a matroid that works for infinite graphs and other infinite structures. For a long time, there were several reasonable proposed axiomatizations of infinite matroids. None of them seemed to cover every important aspect of finite matroid theory. It seemed hard to have one notion of infinite matroids which has bases, circuits, and duality familiar from finite matroids. Let us review the axioms for finite matroids.

For a set $I \subset E$ and $x \in E$, we write $x$, $I + x$, and $I - x$ for $\{ x \}$, $I \cup \{ x \}$ , and $I \setminus \{ x \}$ respectively.

\begin{definition}
Let $E$ be some finite set. A finite matroid $M$ is a pair $(E, \mathcal{L})$ with $\mathcal{L} \subset 2^E$ satisfying the following properties:
\begin{itemize}
	\item 1: $\emptyset \in \mathcal{L}$.
	\item 2: If $B \in \mathcal{L}$ and $A \subset B$, then $A \in \mathcal{L}$.
	\item 3: If $A, B \in \mathcal{L}$ and $|A| < |B| < \infty$, then there exists $b \in B \setminus A$ such that $A+b \in \mathcal{L}$.
\end{itemize}
\end{definition}

The simplest way to generalize this construction is to allow $E$ to be infinite and leave axioms $1$ to $3$ unchanged. In 1976, Welsh introduced this generalization as a $\textit{pre-independence space}$. Welsh's definition of pre-independence spaces can be found in ~\cite{Welsh}.

Pre-independence spaces do not necessarily have maximal independent sets. For example, consider $I=(\mathbb{N}, \mathcal{L} (I))$ with 

\begin{equation*}
\mathcal{L} (I) = \{ S \subset \mathbb{N} \colon |S| < \infty \}.
\end{equation*}

We can check that $I$ satisfies axioms $1$ to $3$ and gives a pre-independence space. It is clear that any independent set must be a proper and finite subset of $\mathbb{N}$ and can be augmented to a larger independent set by adding any additional element. Thus, no element of $\mathcal{L} (I)$ is maximal with respect to set inclusion.

Thus, pre-independence spaces are unsatisfactory as an infinite generalization of finite matroids.

Now, we introduce a set of matroid axioms developed in ~\cite{Axm} in 2010 that captures essential aspects of finite matroid theory while allowing for infinite matroids to be defined. Examples of these infinite matroids that can be created from infinite graphs will be studied in this thesis.

Now we let $E$ to be any set and possibly infinite.
\begin{definition}
A matroid $M$ is a pair $(E,\mathcal{L})$ with $\mathcal{L} \subset 2^{E}$  satisfying the following properties:
\begin{itemize}
	\item I1: $\emptyset \in \mathcal{L}$.
	\item I2: If $B \in \mathcal{L}$ and $A \subset B$, then $A \in \mathcal{L}$.
	\item I3: If $B$ is a maximal element of $\mathcal{L}$ and $A$  is a non-maximal element of $\mathcal{L}$, then there exists $b \in B \setminus A$ such that $A+ b \in \mathcal{L}$.
	\item I4: If $A \in \mathcal{L}$ and $A \subset X \subset E$, then the set $\{S \in \mathcal{L} : A \subset S \subset X\}$ has a maximal element.
\end{itemize}
\end{definition}
Here, we use set inclusion as our partial ordering when we talk about maximality and minimality. Elements of $\mathcal{L}$ are called independent sets. Maximal elements of $\mathcal{L}$ are also called $\textit{bases}$. These first two axioms are familiar from finite matroids. The third axiom is different from our third axiom for finite matroids because there could be infinite independent sets. It may be possible to extend a countable independent set by another countable independent set under these axioms. Thus, cardinality does not give us enough information to determine whether we can extend an independent set by another one. The fourth axiom ensures that every matroid has a base and every matroid minor has a base. A matroid minor will be defined later on.

Elements of $2^E \setminus \mathcal{L}$ are called $\textit{dependent sets}$. A minimal dependent set is called a $\textit{circuit}$. A circuit with only one element is called a $\textit{loop}$. It is possible to define a matroid $M=(E, \mathcal{L})$ by specifying a suitable set of circuits and taking the independent sets to be subsets of $E$ that contain no circuit. Any pair $M=(E, \mathcal{L})$ that satisfies the first two axioms is called an independence system. Although infinite matroids were studied for several decades, this axiom system was formulated as recently as 2010. Under these axioms, every matroid has a base and a dual. Let $M=(E, \mathcal{L})$ be a matroid. We define 
\begin{equation*}
\mathcal{L}^{*} := \{ S \subset E \colon \exists B \in \mathcal{L}^{\mathrm{max}} \text{ s.t. }  S \subset E \setminus  B \}
\end{equation*}
where $\mathcal{L}^{\mathrm{max}}$ is the set of bases of $M$. 
Then $M^{*} = (E, \mathcal{L}^{*})$ is the $\textit{dual}$ of $M$. The authors who developed this axiom system showed that this dual is indeed a matroid. In fact, this duality is an involution and $M^{**} =M$.
A circuit of $M^*$ is called a $\textit{cocircuit}$ of $M$. A loop of $M^*$ is also known as a $\textit{coloop}$ of $M$. Equivalently, a coloop is an element of $M$ that is not contained in any circuit of $M$ and thus contained in every base of $M$.

We will now define matroid minors.

Let $M=(E, \mathcal{L})$ be a matroid. Let $X \subset E$. Define $X^C : = E \setminus X$.  We define $M|X : = M - X^C := (X, \mathcal{L} \cap 2^X)$ as the $\textit{restriction of }M \textit{ to } X$. It is also called the $\textit{deletion of } M \textit{ by } X^C$. We define $M.X : =  M / X^C : =(M^* | X)^*$ and name it the $\textit{contraction of }M \textit{ to } X$. We also call $M.X$ the $\textit{contraction of }M \textit{ by } X^C$. Independent sets of $M.X$ are the sets $I \subset X$ such that $I \cup I' \in \mathcal{L}$ for every independent set $I'$ of $M - X$. Each of these constructions were shown to be matroids in ~\cite{Axm}. The result of any sequence of contractions and restrictions of some matroid $M$ is called a $\textit{minor}$ of $M$.

There are several different ways to axiomatize these structures. Another equivalent axiomatization we will find useful is the following. Consider $M = (E, \mathcal{L})$ with $\mathcal{L} \subset 2^E$ and let $\mathcal{B}:=\mathcal{L}^{\mathrm{max}}$ be alternative notation for the set of maximal elements of $\mathcal{L}$. $\mathcal{B}$ is also known as the set of bases of $M$. A set $S \subset 2^E$ is called $\mathcal{B}$-independent if there is some $B \in \mathcal{B}$ with $S \subset B$. If $\mathcal{B}$ satisfies the following properties, then $M$ is a matroid in the above sense.
\begin{itemize}
	\item B1: $\mathcal{B} \neq \emptyset$.
	\item B2: Whenever $B_1, B_2 \in \mathcal{B}$ and $x \in B_1 \setminus B_2$, there is an element $y$ of $B_2 \setminus B_1$ such that $(B_1 - x)+y \in \mathcal{B}$.
	\item B3: The set of $\mathcal{B}$-independent sets satisfies axiom I4.
\end{itemize}

Axiom B2 is also known as the base exchange axiom.
The following review of historical development borrows from ~\cite{Axm}.

In the 1960s and 1970s, Higgs ~\cite{Bmatroid} and Oxley ~\cite{OxleyI, OxleyII, OxleyIII} developed a different axiom system that described the same class of structures as this one. They named these structures `$B$-matroids'. Although Oxley found a set of axioms for these ‘$B$-matroids’ resembling a mixture of independence and base axioms, it remained an open problem whether an axiom set exists for infinite matroids which allows for minors and duality familiar from finite matroids. 

\begin{definition}
A set system $M=(E, \mathcal{L})$ with $\mathcal{L} \subset 2^E$  is a $B$-matroid if it satisfies the following axioms. 
\begin{itemize}
	\item I1: $\emptyset \in \mathcal{L}$
	\item I2: If $B \in \mathcal{L}$ and $A \subset B$, then $A \in \mathcal{L}$.
	\item B2: Whenever $B_1, B_2 \in \mathcal{B}$ and $x \in B_1 \setminus B_2$, there is an element $y$ of $B_2 \setminus B_1$ such that $(B_1 - x)+y \in \mathcal{B}$.
	\item I4: If $A \in \mathcal{L}$ and $A \subset X \subset E$, then the set $\{S \in \mathcal{L} : A \subset S \subset X\}$ has a maximal element.
\end{itemize}
\end{definition}

A proof that a set system $M=(E,\mathcal{L})$ gives a matroid if and only if it gives a $B$-matroid can be found in ~\cite{Axm}. 

A special class of matroids which are known as the finitary matroids are better understood than general infinite matroids. 
\begin{definition}
A matroid $M$ is $\textit{finitary}$ if a set $S$ is independent in $M$ if and only if all finite subsets of $S$ are independent.
\end{definition}
For every matroid $M = (E,\mathcal{L})$, there exists an associated finitary matroid $M^{\mathrm{fin}}=(E,\mathcal{L}^{\mathrm{fin}})$ whose independent sets are subsets $S$ of $E$ such that every finite subset of $S$ is independent in $M$. The proof of this relies on Zorn's lemma and can be found in ~\cite{Axm}. $M^{\mathrm{fin}}$ is also known as the $\textit{finitarization}$ of $M$. 

To understand finitarization, consider the matroid $M= ( \mathbb{N}, \mathcal{L})$ where

\begin{equation*}
\mathcal{L} := \{ S \subset \mathbb{N} \colon |\mathbb{N} \setminus S| \geq 2 \}.
\end{equation*}

Since all finite subsets of $\mathbb{N}$ are independent in $M$, $M^\mathrm{fin} = ( \mathbb{N}, 2^{\mathbb{N}})$. Here, $M^\mathrm{fin}$ has more independent sets than $M$. We can make this precise in the following way.

Every base $B$ of $M$ extends to a base $F$ of $M^\mathrm{fin}$. To see why, note that any base $B$ is independent in $M^\mathrm{fin}$. Because of our fourth independence axiom for matroids, the set $\{ S \in \mathcal{L} (M^\mathrm{fin}) \colon B \subset S \subset E \}$ has a maximal element $F$. This maximal element is a base of $M^\mathrm{fin}$. 

Conversely, every base $F$ of $M^\mathrm{fin}$ contains a base $B$ of $M$. To see why, consider $F^* := E \setminus F$. Then $F^*$ is a base of $M^{\mathrm{fin}*}$ . Independent sets of $M^{\mathrm{fin}*}$ are also independent in $M^*$. So $F^*$ extends to some base $B^*$ in $M^*$ by the fourth independence axiom. $B := E \setminus B^*$ is then a base in $M$ and $F$ contains $B$.  Unfortunately, the class of finitary matroids is not closed under duality. The authors of ~\cite{Union} define the class of nearly finitary matroids which is still not closed under duality but extends the class of finitary matroids. For more general independence systems, finitarization is still defined but will not necessarily give you a matroid. For example, consider $I:=(\{1,2,3\},\mathcal{L}(I))$ where a set is defined to be independent if it is a subset of $\{1\}$ or a subset of $\{2,3\}$. $I^\mathrm{fin}=I$ is not a matroid. It is possible to have an independence system $N$ that is not a matroid but whose finitarization is nonetheless still a matroid. See Section 4.1 of Chapter 4 for such an example.

The authors of ~\cite{Union} introduced the class of nearly finitary matroids and the class of $k$-nearly finitary matroids.

\begin{definition}
A matroid $M$ is called nearly finitary if whenever a base $F$ in $M^\mathrm{fin}$ contains a base $B$ in $M$, their set difference is a finite set.
\end{definition}

\begin{definition}
For an integer $k$, a matroid is called $k$-nearly finitary if whenever a base $F$ in $M^\mathrm{fin}$ contains a base $B$ in $M$, their set difference has cardinality bounded by $k$.
\end{definition}

\begin{definition}
We say that a matroid is exactly $k$-nearly finitary if it is $k$-nearly finitary but not $(k-1)$-nearly finitary.
\end{definition}

Note that no matroid is $-1$-nearly finitary. Finitary matroids are exactly $0$-nearly finitary. Since matroids are closed under duality, the dual of a finitary matroid is a matroid. However, this dual does not need to be finitary. So finitary matroids are not closed under duality.

We can extend the concepts of nearly finitary and $k$-nearly finitary to independence systems in the natural way as follows.

\begin{definition}
An independence system $I$ is called nearly finitary if whenever a base $F$ in $I^\mathrm{fin}$ contains a base $B$ in $I$, their set difference is a finite set. 
\end{definition}

\begin{definition}
For an integer $k$, a independence system is called $k$-nearly finitary if whenever a base $F$ in $I^\mathrm{fin}$ contains a base $B$ in $I$, their set difference has cardinality bounded by $k$. We say that an independence system is exactly $k$-nearly finitary if it is $k$-nearly finitary but not $(k-1)$-nearly finitary.
\end{definition}

If an independence system $I$ is $k$-nearly finitary for some $k \in \mathbb{N}$, we sometimes write ``$I$ is $k$-nearly finitary" for shorthand. Likewise, if an independence system $I$ is not $k$-nearly finitary for any $k \in \mathbb{N}$, we sometimes write ``$I$ is not $k$-nearly finitary" for shorthand.

We remark that our definitions of nearly finitary and $k$-nearly finitary are slightly different from the ones that appear in ~\cite{Union}. The authors of ~\cite{Union} define an independence system $I$ to be nearly finitary if every base $F$ in $I^\mathrm{fin}$ contains a base $B$ in $I$ such that their set difference is a finite set. For an integer $k$, they define an independence system to be $k$-nearly finitary if every base $F$ in $I^\mathrm{fin}$ contains a base $B$ in $I$ such that their set difference has cardinality bounded by $k$. Before Section 3.1 of Chapter 3, we will show that our definitions of nearly finitary and $k$-nearly finitary are equivalent to the definitions in ~\cite{Union} for matroids. For general independence systems, our definitions are different from those in ~\cite{Union}. In Section 4.1 of Chapter 4, we will show an example of an independence system that is $k$-nearly finitary using the definition in ~\cite{Union} but is not $k$-nearly finitary in our sense.

In this thesis, we primarily study the problem about whether every nearly finitary matroid is $k$-nearly finitary as introduced in ~\cite{Inter}. We henceforth will refer to this problem as problem $P1$. A priori, there could be some nearly finitary matroid with no uniform bound $k$. However, all known examples of nearly finitary matroids are $k$-nearly finitary. There is no known proof that all nearly finitary matroids are $k$-nearly finitary. The motivation for studying problem $P1$ comes from Halin's theorems for infinite graphs in ~\cite{Halin}.

Before we state Halin's theorems, let us review some terminlogy. A graph $G=(V,E)$ is a pair of sets with $E \subset (V \times V)/ \sim$ where $\sim$ is defined by $(v_1,v_2) \sim (v_2,v_1)$. In other words, we do not care about the order of the pair of vertices. We say that $V$ is the vertex set of $G$ and $E$ is the edge set of $G$. Let $(v_1, v_2) \in E$. We will denote this edge by $v_1v_2$ for a shorthand notation. A ray $R$ is an infinite sequence of edges of the form $(v_1v_2, v_2v_3, v_3v_4, \dots )$ such that $v_i \neq v_j$ if $i$ and $j$ are two different integers. We say that two rays $R_1$ and $R_2$ are $\textit{vertex-disjoint}$ if they do not share a vertex. We say that a set of rays $\mathcal{R}$ is $\textit{vertex-disjoint}$ if every pair of rays in $\mathcal{R}$ is vertex-disjoint. Similarly, we say that two rays $R_1$ and $R_2$ are $\textit{edge-disjoint}$ if they do not share an edge. We say that a set of rays $\mathcal{R}$ is $\textit{edge-disjoint}$ if every pair of rays in $\mathcal{R}$ is edge-disjoint. If a set of rays is vertex-disjoint, then it is edge-disjoint. It is possible for a pair of rays to be edge-disjoint without being vertex-disjoint. Consider the graph $G=(\mathbb{Z},E)$ where $E:=\{(i,i+1) \colon i \in \mathbb{Z}\}$. Then $\{ (0,1), (1,2), (2, 3) ... \}$ is edge disjoint from $\{ (-1,0), (-2,-1), (-3,-2) ...\}$. However, these two rays share the vertex $0$. Halin's vertex-disjoint theorem states that if an infinite graph $G$ contains some set of $k$ vertex-disjoint rays for every $k \in \mathbb{N}$, then $G$ contains some set of infinitely many vertex-disjoint rays. Halin's edge-disjoint theorem states that if an infinite graph $G$ contains some set of $k$ edge-disjoint rays for every $k \in \mathbb{N}$, then $G$ contains some set of infinitely many edge-disjoint rays.

To see an example demonstrating Halin's theorems, consider the graph $G$ in Figure 1.1 on page $\pageref{Halin}$. For all $k \in \mathbb{N}$, $G$ has some set of $k$ edge-disjoint rays starting from point $A$. By Halin's edge-disjoint theorem, $G$ has some set of infinitely many edge-disjoint rays. By removing edges adjacent to point $A$, we can find a set of $k$ vertex-disjoint rays in $G$ for all $k \in \mathbb{N}$. By Halin's vertex-disjoint theorem, $G$ has some set of infinitely many vertex-disjoint rays. We can also verify the conclusions of Halin's theorems directly in these examples. 

\begin{figure}
\includegraphics{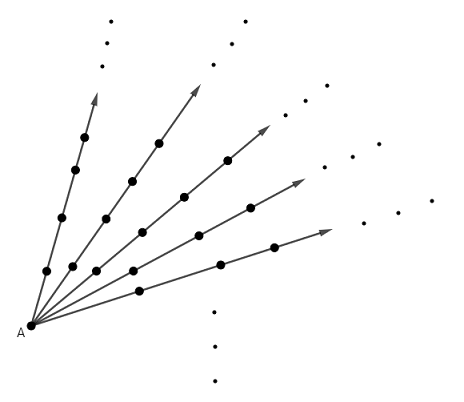}
\caption{Example demonstrating Halin's Theorems} \label{Halin}
\end{figure}

In ~\cite{Inter}, the authors solve problem $P1$ for two special subclasses of infinite matroids. All nearly finitary algebraic cycle matroids and nearly finitary topological cycle matroids are $k$-nearly finitary for some $k$. The main tool used to prove this was Halin's theorem for vertex disjoint rays. It is in this sense that problem $P1$ is a possible generalization of Halin's infinite grid theorem. Problem $P1$ has also been solved affirmatively for several other classes of matroids. In ~\cite{Gammoids}, the authors prove that every nearly finitary gammoid and every nearly finitary transversal matroid are $k$-nearly finitary.

Suppose $B_1$ and $B_2$ are bases of some matroid $M$ and suppose $|B_1| < \infty$. If $B_1=B_2$, then they have the same cardinality. If not, then neither base is a proper subset of the other. So there is some $b_2 \in B_2 \setminus B_1$. Then by using the base exchange axiom B2, there is some $b_1 \in B_1 \setminus B_2$  such that $S_1 := B_2 + b_1  - b_2$ is a base in $M$. It is clear that $|S_1|=|B_2|$. If $S_1 = B_1$, we stop and we know that $|B_2|=|S_1|=|B_1|$. Otherwise, we continue this process to get a new base $S_2$. Since $|B_1| < \infty$, this process will stop after a finite number of steps and we get $|B_2| = |S_n| = |B_1|$ for some finite integer $n$. Thus, all bases of a matroid either have the same finite cardinality or are all infinite in cardinality.

For a matroid $M$, we define $\mathrm{rank} (M) := |B|$ where $B \in \mathcal{L}^{\mathrm{max}}(M)$. We consider all infinite cardinalities equal to make this rank operation well defined.

Before the modern axiom system in ~\cite{Axm} was introduced, there was a different way to axiomatize finitary matroids. As appears in ~\cite{Axm}, we will define classical finitary matroids. 
\begin{definition}
A classical finitary matroid $M$ is a pair $(E,\mathcal{L})$ with $\mathcal{L} \subset 2^{E}$  satisfying the following properties:
\begin{itemize}
	\item F1: $\emptyset \in \mathcal{L}$.
	\item F2: If $B \in \mathcal{L}$ and $A \subset B$, then $A \in \mathcal{L}$.
	\item F3: If $A,B \in \mathcal {L}$ and $|A| < |B| < \infty $, then there exists $b \in B \setminus A$ such that $A+b \in \mathcal{L}$.
	\item F4: $A \in \mathcal{L}$ if and only if all finite subsets of $A$ are in $\mathcal{L}$.
\end{itemize}
\end{definition}

A proof that classical finitary matroids are matroids in our sense appears in ~\cite{Axm}. This proof crucially relies on Zorn's lemma and thus the Axiom of Choice. Those authors even show that these classical finitary matroids are precisely our finitary matroids assuming the Axiom of Choice.

To help us study problem $P1$, we propose a notion of spectrum for matroids. Let $M$ be a matroid and $M^\mathrm{fin}$ be its finitarization. We define the following:

\begin{equation*}
 Spec(M) := \{|F \setminus B|: F \supset B, F \text{ is a base in } M^\mathrm{fin}, \text{and } B \text{ is a base in } M\}.
\end{equation*}

A nearly finitary matroid is $k$-nearly finitary if and only if its spectrum has finite size. In Chapter 3, we will construct nearly finitary matroids with spectrums of arbitrarily large finite size. We also have an example of a matroid with an infinitely large spectrum that is not nearly finitary. Our examples are all based on the algebraic cycle matroid of the one way infinite ladder. It is a matroid with spectrum $\{0,1\}$. We consider the matroid sum of $n$ copies of this matroid for some examples of nearly finitary matroids with large finite spectrums. We consider the matroid sum of infinitely many copies of this matroid for an example of a matroid with infinitely large spectrum.

We also found a sufficient condition for nearly finitary matroids to be $k$-nearly finitary. We prove that a special class of nearly finitary matroids satisfies this condition in Chapter 3.

On the next page there are two graphs of various classes of matroids in Figure 1.2 and Figure 1.3. Arrows indicate class inclusion with smaller classes pointing towards larger classes. Several of these classes are studied extensively in this dissertation.

\newpage
\vfill
\begin{figure}
\begin{tikzpicture}[scale=0.75,transform shape]
  \tikzstyle{VertexStyle}=[rectangle]
  \Vertex[x=0,y=-0.5]{Finite Graph Cycle}
  \Vertex[x=3,y=3]{Topological Cycle}
  \Vertex[x=0,y=3]{Finitary Cycle}
  \Vertex[x=-3,y=3]{Algebraic Cycle}
  \Vertex[x=0, y=5.2]{Psi}
  \Vertex[x=0, y=7]{Partition}
  \tikzstyle{LabelStyle}=[fill=white,sloped]
  \tikzstyle{EdgeStyle}=[post]
  \Edge(Finite Graph Cycle)(Topological Cycle)
  \Edge(Finite Graph Cycle)(Algebraic Cycle)
  \Edge(Finite Graph Cycle)(Finitary Cycle)
  \Edge(Algebraic Cycle)(Psi)
  \Edge(Topological Cycle)(Psi)
  \Edge(Finitary Cycle)(Psi)
  \Edge(Psi)(Partition)
\end{tikzpicture}
\caption{Graphical Matroid Classes}
\end{figure}
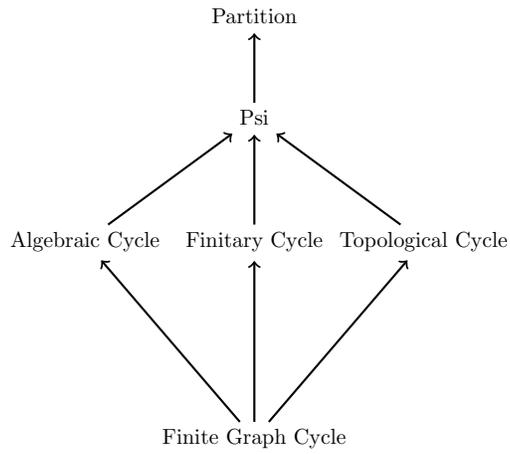

\vfill

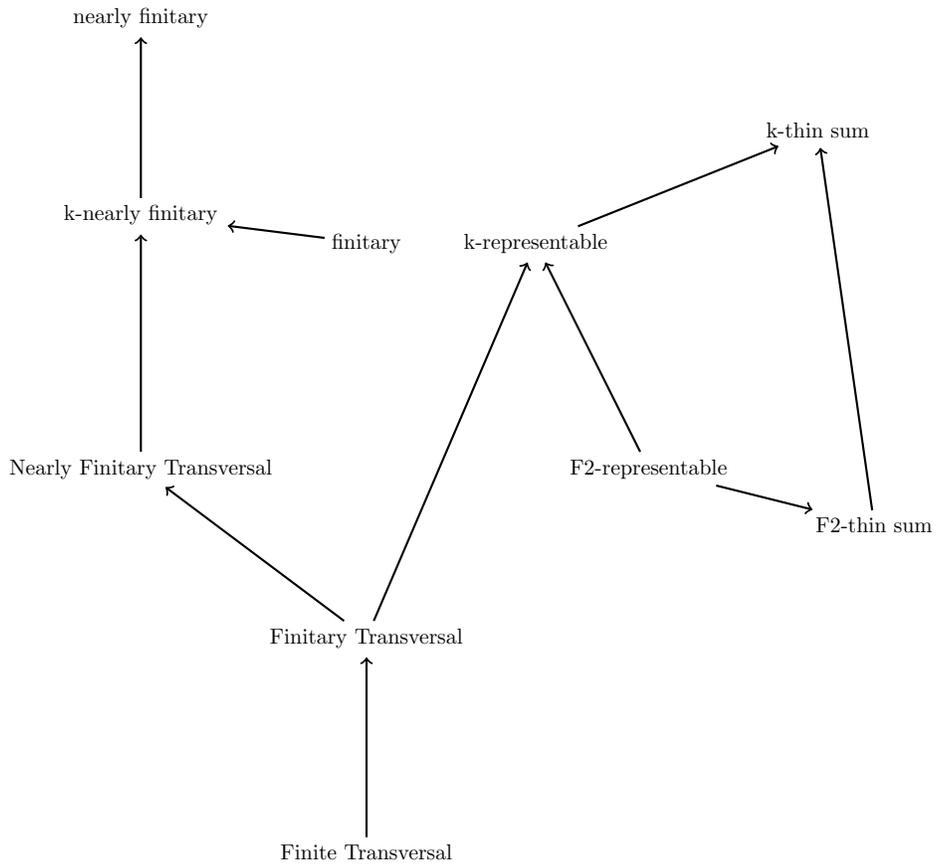
\begin{figure}
\begin{tikzpicture}[scale=0.75,transform shape]
  \tikzstyle{VertexStyle}=[rectangle]
  \Vertex[x=0, y=-3.8]{Finite Transversal}
  \Vertex[x=0,y=0]{Finitary Transversal}
  \Vertex[x=-4,y=3]{Nearly Finitary Transversal}
  \Vertex[x=5,y=3]{F2-representable}
  \Vertex[x=3,y=7]{k-representable}
  \Vertex[x=8,y=9]{k-thin sum}
  \Vertex[x=9,y=2]{F2-thin sum}
  \Vertex[x=0,y=7]{finitary}
  \Vertex[x=-4,y=7.5]{k-nearly finitary}
  \Vertex[x=-4,y=11]{nearly finitary}
  \tikzstyle{LabelStyle}=[fill=white,sloped]
  \tikzstyle{EdgeStyle}=[post]
  \Edge(Finite Transversal)(Finitary Transversal)
  \Edge(Finitary Transversal)(k-representable)
  \Edge(Finitary Transversal)(Nearly Finitary Transversal)
  \Edge(Nearly Finitary Transversal)(k-nearly finitary)
  \Edge(k-nearly finitary)(nearly finitary)
  \Edge(F2-representable)(k-representable)
  \Edge(F2-representable)(F2-thin sum)
  \Edge(k-representable)(k-thin sum)
  \Edge(F2-thin sum)(k-thin sum)
  \Edge(finitary)(k-nearly finitary)
\end{tikzpicture}
\caption{More Matroid Classes}
\end{figure}

\vfill
\clearpage

With an existing matroid $M$, we can generate new matroids in several different ways. These new matroids will give us more examples to consider when we ask whether every nearly finitary matroid is $k$-nearly finitary. One way to generate new matroids will be through the matroid union theorem introduced in ~\cite{Union}. Let $M=(E(M), \mathcal{L}(M))$ and $N=(E(N), \mathcal{L}(N))$ be nearly finitary matroids. Then we can define
\begin{definition}
\begin{equation*}
M \vee N := ( E(M) \cup E(N), \mathcal{L} (M \vee N))
\end{equation*} 
where
\begin{equation*}
\mathcal{L} (M \vee N) := \{ S \cup T \colon S \in \mathcal{L} (M) \text{ and } T \in \mathcal{L} (N) \}.
\end{equation*}
\end{definition}

The theorem tells us that $M \vee N$ is a matroid. This gives us a way to generate large families of nearly finitary matroids from known examples of matroids.

Since we are studying infinite matroids as a generalization of finite matroids, we will also introduce infinite graph theory as a generalization of finite graph theory. Much of our discussion on this matter comes from ~\cite{InfGraph}.

Just as finite graphs give us many examples of finite matroids, infinite graphs will give us many examples of infinite matroids. First, we review one way to get a finite matroid from a finite graph. Let $G:=(E(G),V(G))$ be a finite graph with finitely many edges and finitely many vertices. We can define a matroid $M(G):=(E(G), \mathcal{L} (G))$ where $E(G)$ is the set of edges of $G$ and $\mathcal{L} (G) \subset 2^{E(G)}$ consists of sets of edges that contain no cycle of $G$. Since $G$ is finite, all cycles are finite cycles. This is the $\textit{cycle matroid}$ of $G$. Another kind of matroid we can define is the bond matroid of $G$. Since $G$ is finite, it has some finite number $n_G$ of connected components. Let $M^* (G) := (E(G), \mathcal{L}^* (G))$ where 
\begin{equation*}
\mathcal{L}^*(G) := \{ S \in 2^{E(G)} \colon \text{the graph } (E(G) \setminus S, V) \text{ has } n_G \text{ connected components} \}.
\end{equation*}
We call $M^*(G)$ the $\textit{bond matroid}$ of $G$.
It turns out that $M(G)$ and $M^*(G)$ are dual matroids in this case.

Now we allow $G=(V(G),E(G))$ to be some infinite connected graph and make no assumptions about the cardinalities of the edge and vertex sets. We can define a finitary matroid $M_{fin} (G) := (E(G), \mathcal{L}_{fin} (G))$ where $\mathcal{L}_{fin} (G) \subset 2^{E(G)}$ consists of sets of edges that contain no finite cycle of $G$. Then $M_{fin} (G)$ is a finitary matroid known as the $\textit{finite cycle matroid}$ of $G$. However, because $G$ is infinite there are several possible notions of infinite cycles. One of these notions is known as an algebraic cycle of a graph. We say that a non-empty set of edges $A_C \subset E(G)$ is an $\textit{algebraic cycle}$ if each vertex of the graph $(V(G), A_C)$ has even degree. This is a generalization of a finite cycle since all finite cycles are algebraic cycles. For an example of an algebraic cycle that is not a finite cycle, consider the graph $( \mathbb{Z}, E)$ where

\begin{equation*}
E := \{ ij \colon j=i+1 \in \mathbb{Z} \}.
\end{equation*}

Clearly, $E$ is a nonempty set of edges where each vertex has degree $2$. So $E$ is an algebraic cycle but it is not a finite cycle. For arbitrary graphs $G$, the independence system $M_{AC} (G) := (E(G), \mathcal{L}_{AC} (G))$ where $\mathcal{L}_{AC} (G) \subset 2^{E(G)}$ consists of sets of edges that contain no algebraic cycle of $G$. Unfortunately, $M_{AC} (G)$ is not always a matroid. However, in 1969, Higgs gave a necessary and sufficient condition for when $M_{AC} (G)$ is a matroid in ~\cite{Higgs}. Recall that it is possible to define an independence system by specifying its set of circuits. The set of circuits of an algebraic cycle system of a graph $G$ consists of the elementary algebraic cycles of $G$. An $\textit{elementary algebraic cycle}$ $C_A$ is an algebraic cycle such that if you remove any edge from $C_A$, the new edge set no longer contains an algebraic cycle. We use the term 'elementary' in this way for other kinds of cycles too.

\begin{theorem}
( ~\cite{Higgs} Theorem 5.1)
The elementary algebraic cycles of an infinite graph $G$ are the circuits of a matroid on its edge set $E(G)$ if and only if $G$ contains no subdivision of the Bean graph.
\end{theorem}

Figure 1.4 is a picture of the Bean graph mentioned in Higgs' theorem and seen in ~\cite{InfGraph}.

\begin{figure}
\includegraphics{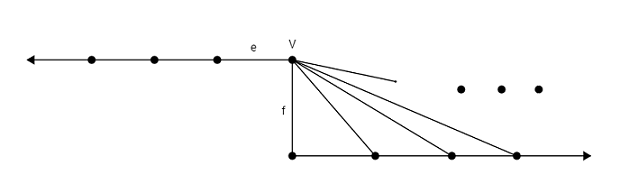}
\caption{Bean Graph}
\label{BeanGraph}
\end{figure}

Let us denote the Bean graph by $B$. To see why $M_{AC} (B)$ is not a matroid, consider the following. Let $H$ be the set of horizontal edges of $B$. $H$ is a maximal independent set in $M_{AC} (B)$. Let $H_f := (H-e)+f=(H \setminus \{ e \}) \cup \{ f \}$. This is another maximal independent set of $M_{AC} (B)$. Now consider the set $S$ which consists of all top horizontal edges of $B$ and all non horizontal edges of $B$. $S$ is a maximal independent set of $M_{AC} (B)$. $S-e = (S \setminus \{ e \} )$ is non maximal. You cannot add any edge of $H_f$ to $S-e$ to get another independent set of $M_{AC} (B)$. So our third independence axiom fails. Higgs' theorem shows that this is essentially the only example of a graph whose algebraic cycle system is not a matroid. Knowing when an independence system is a matroid is important to the study of problem $P1$. After all, problem $P1$ only concerns matroids and not other independence systems.

Another possible notion of an infinite cycle is known as a topological cycle and was studied in ~\cite{Top}. Before we define the notion of topological cycle, we will define what an end $\omega$ of an infinite graph is.

Let $R_1$ and $R_2$ be two rays of some graph $G=(V(G), E(G))$. We say that $R_1 \sim R_2$ if for every finite set $S$ of edges of $G$, $R_1$ and $R_2$ are in the same connected component of $(E(G) \setminus S, V(G))$. This induces an equivalence relation on the set of rays of $G$. An end $\omega$ is then defined to be an equivalence class of rays under $\sim$. For each end $\omega$ of $G$, we add a point at infinity corresponding to that end. This is known as the $\textit{End compactification}$ (or $\textit{Freudenthal compactification}$) of $G$. Given a graph $G$, together with an end boundary, a $\textit{topological cycle}$ is a homeomorphic image of the unit circle in the topological space consisting of the graph together with the boundary. This topological space will be denoted by $|G|$.

A double ray is a two way infinite sequence of edges $(...v_{-2}v_{-1}, v_{-1}v_0, v_0v_1, v_1v_2...)$ such that $v_i \neq v_j$ if $i$ and $j$ are two different integers. Removing one edge from a double ray will give you a pair of vertex-disjoint rays.

Let $D$ be a double ray and let $d$ be an edge of $D$. We say that $D$ belongs to an end $\omega$ if the two rays of $D \setminus \{ d \}$ are each an element of $\omega$. The $\textit{elementary topological cycles}$ of $G$ are elementary finite cycles and double rays belonging to some end of $G$.

The question about when the elementary topological cycles form the circuits of a matroid was explored in ~\cite{Top}. There, Carmesin proves that topological cycles of $|G|$ induce a matroid if and only if $G$ does not have a subdivision of the dominated ladder as shown in Figure 1.5.

\begin{figure}
\includegraphics{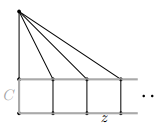}
\caption{Dominated Ladder}
\end{figure}

We will introduce other possible notions of cycles in Chapter 5.

One well known way to construct finite matroids on some finite ground set $E$ is to use the well known notion of representability. Let $V$ be a vector space and consider a function $\phi \colon E \rightarrow V$. Then we can define a matroid $M( \phi ) = (E, \mathcal{L})$ where
\begin{equation*}
\mathcal{L} := \{ S \subset E \colon \phi (S) \text{ is linearly independent in } V \}.
\end{equation*}

We call $M$ a representable matroid. If $V$ is a vector space over the field $k$, we say that $M$ is $k$-representable. Many examples of finite matroids turn out to be representable. The cycle and bond matroids of a finite graph are both representable. In fact, the cycle matroid of a finite graph $G$ is dual to the bond matroid of $G$. More generally, the class of finite representable matroids is closed under duality.

This notion of representability still works using algebraic linear independence if we allow $E$ to be infinite and $V$ to be infinite dimensional. However, the sums considered in algebraic linear independence are only defined if there are finitely many nonzero summands. Thus, a set $\phi (S)$ of vectors is linearly independent in the algebraic sense if and only if all of its finite subsets are independent. Thus, representability is an inherently finitary concept. Moreover, the dual of a representable infinite matroid need not be representable. To see this, consider the matroid $M ( \phi )= (\mathbb {N}, \mathcal{L} (M))$ where $\phi (1) \neq 0$ and $\phi (n) = \phi (1) $ for all $n \in \mathbb{N}$. Then $M( \phi )$ is a representable matroid by construction. Its independent sets are the subsets of size one. Thus, we have:

\begin{equation*}
[M ( \phi )]^* = \{ S \subset \mathbb{N} \colon \mathbb{N} \setminus S \neq \emptyset \}. 
\end{equation*}

Here, $\mathbb{N}$ is not independent but all of its finite subsets are independent. So $[M ( \phi )]^*$ is not finitary and thus not representable. 

To study non-finitary matroids, we will need to use the more general notion of thin representability that was introduced in ~\cite{InfGraph} and studied in ~\cite{ThinShum}. Our discussion of thin sums systems is based on the work in ~\cite{ThinShum}.

We denote the set of functions from $A$ to $k$ by $k^A$. Let $\mathcal{F}_E$ be a family of functions from $k$ to $A$ indexed by $E$.  We say that the family $\mathcal{F}_E$ is $\textit{thin}$ if for each $a \in A$, there are only finitely many $e \in E$ such that $f(e)(a) \neq 0$. Let $E$ be a set and consider a function $f \colon E \rightarrow k^A$ that gives us a family of functions $f(E)$. We say that a set $E' \subset E$ is thin if $f (E')$ is a thin family. A $\textit{thin dependence}$ of $E$ is a map $c \colon E \rightarrow k$ such that for each $a \in A$

\begin{equation*}
\sum_{e \in E} c(e) f(e) (a) = 0.
\end{equation*}

Even if there are only finitely many nonzero summands for each particular $a$, there could be infinitely many $e$ such that there is some $a \in A$ with $c(e)f(e)(a) \neq 0$. In that case, the sum $\sum_{e \in E} c(e)f(e)$ would not be well defined even though it is defined pointwise. Thus, this new notion of thin dependence is not the same as linear dependence. We also say that $c$ is a thin dependence of a subset $E'$ of $E$ if $c(e)$ is the zero outside of $E'$. The trivial thin dependence on $E$ has $c(e)=0$ for all $e \in E$. A subset $E' \subset E$ is called $\textit{thinly independent}$ if there is no nontrivial thin dependence of $E'$. Note that a thinly independent set $E'$ does not necessarily induce a thin family of functions. We say that $E'$ is $\textit{thinly dependent}$ if it is not thinly independent. We can define an independence system $M=(E, \mathcal{L}(M))$ with $\mathcal{L}(M)$ where $\mathcal{L}(M)$ consists of the thinly independent subsets of $E$. We say that $M$ is a $\textit{thin sums system}$. When $M$ is a matroid, we say $M$ is a $\textit{thin sums matroid}$.

The infinite independence systems we get from a graph's algebraic and topological cycles are thinly representable over any field in the following ways.

Let $G=(V(G),E(G))$ be a graph and consider the algebraic cycle system $M_{AC} (G)$. Put an arbitrary orientation on each edge to make $G$ a digraph. Let $k$ be any field. For each $e \in E(G)$, we can define a function $f(e) \in k^{V(G)}$ defined as follows. Set $f(e)(v) = 1$ if $e$ originates from $v$, $f(e)(v)=-1$ if $e$ terminates at $v$, and $f(e)(v)=0$ if $e$ does not touch $v$. Then we get an independence system $M_{thinAC}(G)=(E(G), \mathcal{L} (M))$ using the thinly independent subsets of $E(G)$.

\begin{theorem}
(~\cite{ThinShum} Proposition 2.9)
$M_{thinAC}(G) = M_{AC}(G)$.
\end{theorem}

\begin{proof}
We want to show that $D$ is dependent in $M_{thinAC}(G)$ if and only if it is dependent in $M_{AC}(G)$. Suppose $D$ is dependent in $M_{AC}$. Then $D$ contains an elementary algebraic cycle. That is, $D$ contains an elementary finite cycle or double ray. Let $D' \subset D$ be the edge set of this cycle or double ray. Give a direction to $D'$. For any edge $e \in D$, define $c(e)$ to be $1$ if $e$ is an edge of $D'$ and the direction given to $e$ by the digraph $G$ is the same as the direction given to $e$ by the direction of $D'$ , $-1$ if $e$ is an edge of $D'$ and the two directions given to $e$ are opposites, and $0$ if $e \notin D'$. For every vertex $v \in V(G)$, $\sum_{d \in D'}c(e)f(e)(v)=0$, so $c$ is a thin dependence of $D$. Conversely, suppose that $D$ is dependent in $M_{thinAC}(G)$. Whenever a vertex $v$ is an end of an edge in $D$, it has to be the end of at least two edges in $D$. For each vertex $v$ that touches some edge in $D$, pick two edges in $D$ that touch $v$ to get a subset $D' \subset D$. $D'$ is an algebraic cycle.
\end{proof}

In ~\cite{ThinShum}, the authors show that matroids thinly representable by a thin family are dual to a representable matroid that is finitary. Since we are studying nearly finitary matroids, it is natural to ask if there is a dual notion of nearly thin families for thin sum matroids. We define a family of functions $\mathcal{F}_E$ from $k$ to $A$ to be $n$-$\textit{nearly thin}$ if there are at most $n$ elements $a \in A$ such that there are infinitely many $e \in E$ such that $f(e)(a) \neq 0$. In Chapter 6, we will introduce examples of matroids from $n$-nearly thin families that are dual to $n$-nearly finitary matroids. We conjecture that matroids from $n$-nearly thin families are dual to $n$-nearly finitary matroids. We also conjecture if we have a thin sums matroid that is not thinly representable by an $n$-nearly thin matroid, then it is not dual to a nearly finitary matroid.

Finally, we introduce a new notion of representability related to thin sums and study it in the end of Chapter 6.

   \chapter[%
      Short Title of 2nd Ch.
   ]{%
     Summary of Main Results
   }%
   \label{ch:2ndChapterLabel}
   In Chapter 3, we study nearly finitary matroids in various ways. We first introduce a notion of finitarization spectrum and figure out a few results pertaining to this notion. Later on, we find a condition that guarantees a nearly finitary matroid is $k$-nearly finitary when satisfied. We were unable to find nearly finitary matroids that do not satisfy this condition. We study this condition in detail and prove some results related to this condition. In particular,  $(M^\mathrm{fin*} \vee M)^*$ from below is sometimes the same as $\hat{M}$ from below. Then, we assume the existence of a nearly finitary matroid that is not $k$-nearly finitary. From this assumption, we derive the existence of several families of nearly finitary matroids that are not $k$-nearly finitary. Finally, we show that the union of a finite rank matroid with any other matroid gives us a matroid.

\begin{customthm}{3.2.1}
There exists nearly finitary matroids with finitarization spectrum of arbitrarily large finite size.
\end{customthm}

\begin{customthm}{3.4.1}
Let $M=(E,\mathcal{L})$ and $M^{\mathrm{fin}}=(E,\mathcal{L}^{\mathrm{\mathrm{fin}}})$ be a matroid and its finitarization respectively. Define $\hat{M}:=(E,\mathcal{K})$ where $S\in\mathcal{K}$ if there exists a base $F$ in $M^{\mathrm{fin}}$ and a base $B$ in $M$ such that $B\subset F$ and $S\subset F\setminus B$. 

If $M$ is a nearly finitary matroid such that $\hat{M}$ is a matroid, then $M$ is $k$-nearly finitary.
\end{customthm}

\begin{customthm}{3.4.3}
Suppose $M$ is a nearly finitary and nearly cofinitary matroid. Then $(M^\mathrm{fin*} \vee M)^*$ is a matroid.
\end{customthm}

We propose the following independence system. Let $M$ be a nearly finitary matroid and define the following
\begin{equation*}
S(M):= \{ F^* \cup B \colon F^* \text{ is a base of } M^\mathrm{fin*}, B \text{ is a base of } M, \text{ and } F^* \cap B = \emptyset \}. 
\end{equation*}

We then define $S(M)_\mathrm{min}$ to be the minimal elements of $S(M)$ with respect to set inclusion.
\begin{customthm}{3.4.4}
$S(M)_\mathrm{min}$ is non-empty if $M$ is a $k$-nearly finitary matroid.
\end{customthm}

The rank of a matroid $M$ is the cardinality of a base in $M$. For a matroid $M$ with rank at least $k$, a related matroid $M[k]$ is defined in the beginning of Section 3.3. This definition is originally from ~\cite{Union}.
\begin{customthm}{3.4.5}
Suppose that there exists some matroid $M=(E(M),\mathcal{L} (M))$ that is nearly finitary but not $n$-nearly finitary. Then $M[k]$ is also nearly finitary but not $n$-nearly finitary.
\end{customthm}

\begin{customthm}{3.4.6}
Let $N=(E(N), \mathcal{L}(N))$ be a nearly finitary matroid with $E(N)$ disjoint from $E(M)$. By the nearly finitary matroid union theorem in ~\cite{Union}, the matroid union $M \vee N$ is a nearly finitary matroid. Moreover, it is not $n$-nearly finitary.
\end{customthm}

\begin{customthm}{3.5.1}
Suppose that $F$ is a finite rank matroid and $M$ is any matroid. Then the union $F \vee M$ is a matroid.
\end{customthm}

In Chapter 4, we introduce a notion of near-finitarization. More importantly, we answer a certain generalization of problem $P1$ negatively.

\begin{customthm}{4.2.1}
There exists a nearly finitary independence system that is not $k$-nearly finitary.
\end{customthm}

In Chapter 5, we study $\Psi$-independence systems from ~\cite{Games}. We define $P (\Psi)$-matroids which generalize $\Psi$-matroids and prove the following theorem.
\begin{customthm}{5.2.1}
Suppose $M=(E, \mathcal{L})$ is a nearly $P( \Psi )$-matroid on a graph $G$ with finitely many disjoint rays. Then $M$ is $k$-nearly finitary.
\end{customthm}

We also study $\Psi$-finite tree systems from ~\cite{Games}. Under the axiom of determinacy, all $\Psi$-finite tree systems are matroids.
\begin{customthm}{5.3.1}
It is impossible to pick an axiom system such that all $\Psi$-finite tree systems are matroids and all classical finitary matroids are matroids.
\end{customthm}

In Chapter 6, we introduce the notion of a nearly thin family and speculate that it is in some sense dual to the notion of nearly finitary. We introduce an example showing why we believe this. Finally, we introduce a new kind of independence system inspired by thin sums representability and we give a condition for when an independent set in this kind of system is maximal.

\begin{customthm}{6.2.1}
Let $M_f$ be a topological independence system. An independent set $S$ of $M_f$ is maximal if and only if for all $e\in E$, 
\begin{equation*}
f(e)\in\hat{f}\left(C_{S}(f)\right).
\end{equation*}
\end{customthm}

   \chapter[%
      Short Title of 3rd Ch.
   ]{%
     Finitarization Spectrum
   }%
   \label{ch:3rdChapterLabel}
   \begin{section}{Definitions and Motivation}

Earlier, we introduced a notion of finitarization spectrum for matroids. Recall that
\begin{equation*}
 Spec(M) := \{|F \setminus B|: F \supset B, F \text{ is a base in } M^\mathrm{fin}, \text{and } B \text{ is a base in } M\}.
\end{equation*}
For our considerations, we consider all infinite cardinalities equal. 

We can also define a dual notion of spectrum with $coSpec(M) := Spec(M^*)$.

A nearly finitary matroid is $k$-nearly finitary if and only if its spectrum has finite size. This leads us to consider what possible sizes a matroid spectrum can have. We construct nearly finitary matroids with spectrums of any finite cardinality. We also construct a matroid with an infinitely large spectrum that is not nearly finitary. We will start by reviewing a matroid that inspired our definition of spectrum since its spectrum is not a singleton set. Consider the algebraic cycle matroid $M=(E,\mathcal{L})$ of the one way infinite ladder graph $G$ that will be drawn below in Figure 3.1.

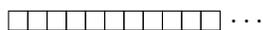
\begin{figure}
\begin{tikzpicture}[
      start chain=1 going right, node distance=-0.15mm
    ]
    \foreach \x in {1,2,...,11} {
        \x, \node [draw,on chain=1] {};
    } 
    \node [name=r,on chain=1] {\ldots}; 
\end{tikzpicture}
\caption{One Way Infinite Ladder}
\label{Ladder}
\end{figure}

Recall that a $\textit{spanning tree}$ of a graph is a connected set of edges that touch every vertex and contains no finite cycle. For any infinite graph $G=(V(G),E(G))$, we can construct a topological space $\hat{G}$ associated to $G$ where $V(G)$ is a totally disconnected subspace of $\hat{G}$. For each edge $v_iv_j \in E(G)$, we have a homeomorphism $\phi_{ij}$ from the interval $[0,1]$ to $\hat{G}$ given by $\phi_{ij} (0) = v_i$ and $\phi_{ij} (1) = v_j$. Any edge set $S$ of $G$ corresponds to the topological subspace $\hat{S}$ of $\hat{G}$ that includes the vertices of $S$ and points in between $v_i$ and $v_j$ if $v_iv_j \in S$. To help us define an algebraic spanning tree, we consider the one point compactification of $\hat{G}$ by adding a point $\omega_G$ at infinity if necessary. We denote this space by $\hat{G}_{\mathrm{algebraic}}$. Each edge set $S$ of $G$ corresponds to some topological subspace $\hat{S}_{\mathrm{algebraic}}$ of $\hat{G}_{\mathrm{algebraic}}$. If $S$ contains no ray, then $\hat{S}_{\mathrm{algebraic}} = \hat{S}$. If $S$ contains any ray, then $\hat{S}_{\mathrm{algebraic}} = \hat{S} \cup \{ \omega_G \}$. In other words, any ray in $G$ touches $\omega_G$. We define a set of edges $S$ to be algebraically connected if the topological space $\hat{S}_{\mathrm{algebraic}}$ is connected. Thus, any set of two vertex-disjoint rays would be algebraically connected in this graph. We define an $\textit{algebraic spanning tree}$ to be an algebraically connected set of edges that touch every vertex and contains no algebraic cycle. Bases of this matroid are algebraic spanning trees of $G$. The finitarization of $M$ is the finite cycle matroid of $G$ whose circuits are elementary finite cycles of $G$. Bases of $M^\mathrm{fin}$ are ordinary spanning trees.

The following theorem is inspired by an exercise left to readers.

\begin{theorem}
(See ~\cite{Union} page 2)
Let $B$ be a base in $M$. Suppose $F_1$ and $F_2$ are bases in $M^\mathrm{fin}$ that contain $B$. Then $|F_2 \setminus B|=|F_1 \setminus B|$.
\end{theorem}
\begin{proof}
We will use proof by contradiction. Without loss of generality, assume that $|F_2 \setminus B| > |F_1 \setminus B|$. Since we consider all infinite cardinalities equal, this means that $|F_1 \setminus B|$ is finite. $M^{\mathrm{fin}}$ is a matroid and thus satisfies the base exchange axiom. Since $F_2$ has more elements outside $B$ than $F_1$, there exists $x \in F_2 \setminus F_1$. Since $B \subset F_1$, $x \notin B$. By basis exchange, there is some $y \in F_1 \setminus F_2$ such that $F_3 := (F_2 -x) + y $ is a base in $M^\mathrm{fin}$. $F_3$ still contains $B$ since $x \notin B$. Moreover, 

\begin{equation*}
|F_3 \setminus B| = |F_2 \setminus B| > |F_1 \setminus B|.
\end{equation*}

So we can continue inductively using base exchange to get new bases of $M^\mathrm{fin}$. Also notice that

\begin{equation*}
|F_1 \setminus F_3| < |F_1 \setminus F_2| \leq |F_1 \setminus B| < \infty.
\end{equation*}

Since $|F_1 \setminus B|$ is finite, we can continue this process and eventually have some base $F \supset F_1$ of $M^{fin}$ with $B \subset F$ and

\begin{equation*}
|F \setminus B| > |F_1 \setminus B|.
\end{equation*}

So $F_1$ must be a proper subset of $F$ but this contradicts the maximality of $F_1$. We conclude that $|F_2 \setminus B|=|F_1 \setminus B|$ and we see that any base $B$ in $M$ can contribute at most one element to the spectrum of $M$.
\end{proof}

We can also fix a base $F$ of $M^\mathrm{fin}$ and vary bases of $M$ to get a similar theorem.

\begin{theorem}
Suppose that $F$ is a base in $M^\mathrm{fin}$ and suppose that $B_1$ and $B_2$ are bases in $M$ contained in $F$. Then $|F \setminus B_2| = |F \setminus B_1|$.
\end{theorem}
\begin{proof}
Without loss of generality, assume that $|F \setminus B_1| > |F \setminus B_2|$. Like before, this means that $|F \setminus B_2|$ is finite. Since $B_2$ has more elements inside $F$ than $B_1$, there exists  $x \in B_2 \setminus B_1$. So there is some $y \in B_1 \setminus B_2$ such that $B_3 : = (B_2 -x)+y$  is a base in $M$. $B_3$ is still contained in $F$ since $y \in F$. Moreover,

\begin{equation*}
|F \setminus B_3| = |F \setminus B_2| < |F \setminus B_1|.
\end{equation*}

So we can continue inductively using base exchange to get new bases of $M$. Also notice that

\begin{equation*}
|B_1 \setminus B_3| < |B_1 \setminus B_2| \leq |F \setminus B_2| < \infty.
\end{equation*}

So we can continue this process and eventually have some base $B \supset B_1$ of $M$ with $B \subset F$ and
\begin{equation*}
|F \setminus B_1| > |F \setminus B|.
\end{equation*}

So $B_1$ must be a proper subset of $B$ but this contradicts the maximality of $B_1$. So $|F \setminus B_1| = |F \setminus B_2|$ and we see that any base $F$ of $M^\mathrm{fin}$ can contribute at most one element to the spectrum of $M$. This also shows that our definitions of nearly finitary matroid and $k$-nearly finitary matroid coincide with the definitions in ~\cite{Union}.
\end{proof}
\end{section}

\begin{section}{Ladders}
Let $B$ be the set of all top edges and all middle edges of the one way infinite ladder. Then $B$ contains no algebraic cycle and is an algebraic spanning tree. $B$ is also a spanning tree. Thus, $B$ is a base in both $M$ and its finitarization. This shows that $0 \in Spec(M)$.
Now let $A$ be the set of all top edges and all bottom edges. Because of the point at infinity, $A$ is algebraically connected. It contains no algebraic cycle and is an algebraic spanning tree. However, $A$ is not connected in the usual sense so it is not a spanning tree. Thus, $A$ is not a base of $M^\mathrm{fin}$. Adding any middle edge to $A$ will make it a spanning tree and a base in $M^\mathrm{fin}$. This shows that $Spec(M)$ contains $\{0,1\}$. Suppose there is some algebraic spanning tree $T$ that is not a spanning tree. Since $T$ is a base in $M$, adding any edge to $T$ will make the new set contain an algebraic cycle. Suppose there is a spanning tree $S$ that contains $T$ and at least two additional edges $x$ and $y$.  $T \cup \{x\}$ will contain an infinite algebraic cycle (i.e. a double ray) since it cannot contain any finite cycle. $T \cup \{y\}$ will contain a different double ray. $S$ will contain both double rays. In this particular graph $G$, the union of any two double rays contains a finite circuit. Thus, $Spec(M)$ cannot contain $2$ or any other larger number. We conclude that $Spec(M)=\{0,1\}$. We use this basic example to prove the following.

\begin{theorem}
There exists nearly finitary matroids with finitarization spectrum of arbitrarily large finite size.
\end{theorem}

\begin{proof}
Consider a graph $G_n$ of $n$ parallel one way infinite ladders as shown above in Figure 3.2 for $n=3$. 

\begin{figure}
\begin{tikzpicture}[
      start chain=1 going right, node distance=-0.15mm
    ]
    \foreach \x in {1,2,...,11} {
        \x, \node [draw,on chain=1] {};
    } 
    \node [name=r,on chain=1] {\ldots}; 
\end{tikzpicture}

\begin{tikzpicture}[
      start chain=1 going right, node distance=-0.15mm
    ]
    \foreach \x in {1,2,...,11} {
        \x, \node [draw,on chain=1] {};
    } 
    \node [name=r,on chain=1] {\ldots}; 
\end{tikzpicture}

\begin{tikzpicture}[
      start chain=1 going right, node distance=-0.15mm
    ]
    \foreach \x in {1,2,...,11} {
        \x, \node [draw,on chain=1] {};
    } 
    \node [name=r,on chain=1] {\ldots}; 
\end{tikzpicture}
\caption{3 One Way Infinite Ladders}
\label{Three}
\end{figure}
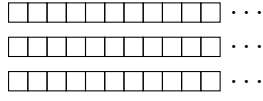

Let $M_n$ be the algebraic cycle matroid of $G_n$. Bases of $M_n$ are algebraic spanning forests of $G_n$. Let $B_i$ be the set of all top edges and middle edges of the $i$-th connected component of $G_n$. Index the connected components of $G_n$ by $C_i$ where $1 \leq i \leq n$. Let $A_i$ be the set of all top edges and bottom edges of the $i$-th connected component of $G_n$. Let $0 \leq k \leq n$. Let $A^{k} =\bigcup_{i=1}^{k}A_{i}$ and $B^{k} =\bigcup_{i=k+1}^{n}B_{i}$. $A^{0}$ and $B^{n}$ are defined to be empty. Then $A^{k} \cup B^{k}$  is a base in $M_n$. Add an edge to each connected component of $A^{k}$ to get a new set $A^{k}_2$. Connected components of $B^{k}$ do not need any added edges since any such added edge will create a finite cycle. So $A^{k}_2 \cup B^{k}$ is a base in the finitarization of $M_n$. This shows that $k$ is in the spectrum for $0 \leq k \leq n$. So $Spec(M)$ contains $\{0,1,2,...,n\}$. For each connected component of $G_n$, an algebraic spanning tree has at most one edge missing from a spanning tree. Thus, any base of $M_n$ can be extended to a base in the finitarization by adding at most $n$ elements. So $Spec(M_n)=\{0,1,2,...,n\}$.
\end{proof}

We can also construct these examples involving multiple ladders using the Nearly Finitary Matroid Union Theorem in ~\cite{Union} and the example involving a single ladder from Section 3.1.

Given a matroid $M$, a natural question to consider is the following. Suppose you have two bases $S_1$ and $S_2$  of $M$ that are respectively contained in bases $F_1$ and $F_2$ of $M^\mathrm{fin}$ with $|F_1 \setminus S_1|=|F_2 \setminus S_2|< \infty$. Must there exist some base $F$ of $M^\mathrm{fin}$ that contains both $S_1$ and $S_2$? The examples that we have considered in this section answer this question negatively.  Consider $M_2$. Consider $S_1=A_1 \cup B_2$  and $S_2 = A_2 \cup B_1$. $S_1$ and $S_2$ are both bases in $M_2$ contained respectively in bases $F_1$ and $F_2$ of $M_2^\mathrm{fin}$ with $|F_1 \setminus S_1| = |F_2 \setminus S_2|=1$. However, no base $F$ of $M_{2}^\mathrm{fin}$ contains both $S_1$ and $S_2$.

Conversely, we can also show that there are two bases $F_1$ and $F_2$ of $M_2^\mathrm{fin}$ that respectively contains $S_1$ and $S_2$ with $|F_1 \setminus S_1|=|F_2 \setminus S_2|< \infty$ such that no base $S$ of $M_2$ is contained in both $F_1$ and $F_2$. Let $F_1 = A_1 \cup B_2 + e_1$ and $F_2 = A_2 \cup B_1 + e_2$ where $e_i$ is the left-most middle edge of the $i$-th connected component of $G_2$. Then $F_1$ and $F_2$ are bases in $M_{2}^\mathrm{fin}$ that show that $1 \in Spec( M_2 )$. However, $F_1 \cap F_2$  cannot contain any algebraic spanning forest of $G_2$. Thus, no base of $M_2$ is contained in both $F_1$ and $F_2$.

Let us consider a graph $G$ of countably infinite parallel one way infinite ladders. Since $G$ has no subdivision of the Bean graph, the algebraic cycles of $G$ induce a matroid. Let $M_{\infty}$ be the algebraic cycle matroid of $G$.  Like when we studied graphs consisting of finitely many ladders, let $B_i$ be the set of all top edges and middle edges of the $i$-th connected component of $G$ and let $A_i$ be the set of all top edges and bottom edges of the $i$-th connected component of $G$. Define $A^0$ to be empty. Let $k \geq 0$. Let $A^{k} =\bigcup_{i=1}^{k}A_{i}$ and $B^{k} =\bigcup_{i=k+1}^{\infty}B_{i}$. Then like when we had finitely many ladders, $A^{k} \cup B^{k}$ gives us a base of $M_{\infty}$ where you need to add $k$ elements to get a base in the finitarization. This shows that $k$ is in the spectrum for any natural number $k$. Define $A^{\infty}=\bigcup_{i=1}^{\infty}A_i$. Then $A^{\infty}$  will be a base where infinitely many edges must be added to get a base in the finitarization. This shows us that $\infty$ is in the spectrum. Thus, $Spec(M_{\infty}) = \mathbb{N} \cup \{\infty\}$. Although the spectrum of $M_{\infty}$ is infinitely large, this matroid is not nearly finitary.

\end{section}

\begin{section}{More on Spectrum}

It seems that all examples of nearly finitary spectrums we have considered so far are either singleton sets or have a set of consecutive integers starting from $0$. This motivates us to construct matroids with different kinds of spectrums.

Recall that the rank of a matroid $M$ is the cardinality of a base $B$ in $M$ with infinite cardinalities considered equal. Let $\mathcal{M}$ be the class of matroids and let $k$ be a natural number. Let $\mathcal{M}_{\geq k}$ be the class of matroids of rank at least $k$. For $k \in \mathbb{N}$, the authors of ~\cite{Union} define a map $[k]$ from $\mathcal{M}_{\geq k}$ to $\mathcal{M}$ that sends $M=(E,\mathcal{L}(M))$ to $M[k]=(E,\mathcal{L}[k])$ where:   

\begin{equation*}
\mathcal{L}[k] = \{ S \in \mathcal{L} : \text{ there exists } T \in \mathcal{L} \text{ such that } T \supset S \text{ and } |T \setminus S| = k \}.
\end{equation*}

Keep in mind that $I+y := I \cup \{ y \}$ and $I-y:= I \setminus \{ y \}$.

\begin{proposition}
(~\cite{Union} Proposition 4.13)
If $M$ is a matroid of rank at least $k$, then $M[k]$ is a matroid.
\end{proposition}

\begin{proof}
Since $rank(M) \geq k$, axioms $I1$ and $I2$ hold. For axiom $I3$, suppose $I$ and $I'$ are independent in $M[k]$ with $I'$ maximal and $I$ non maximal. There is a set $F' \subset E(M) \setminus I'$ of size $k$ such that, in $M$, the set $I' \cup F'$ is not only independent but, by maximality of $I'$, also a base. Similarly, there is a set $F \subset E(M) \setminus I'$ of size $k$ such that $I \cup F \in \mathcal{L} (M)$.

	We claim that $I \cup F$ is non-maximal in $\mathcal{L} (M)$ for any such $F$. Suppose $I \cup F$ is maximal for some $F$ as above. By assumption, $I$ is contained in some larger set of $\mathcal{L} (M[k])$. Hence there is a set $F^{+} \subset E(M) \setminus I$ of size $k+1$ such that $I \cup F^{+}$ is independent in $M$. Clearly $(I \cup F) \setminus (I \cup F^{+}) = F \setminus F^{+}$ is finite. So Lemma 3.3.1 implies that
\begin{equation*}
|F^{+} \setminus F| = | (I \cup F^{+}) \setminus (I \cup F) | \leq | ( I \cup F ) \setminus (I \cup F^{+})| = |F \setminus F^{+}|.
\end{equation*}
In particular, $k+1= |F^{+}| \leq |F| = k$, a contradiction.
	
	Hence we can pick $F$ such that $F \cap F'$ is maximal and, as $I \cup F$ is non-maximal in $\mathcal{L} (M)$, apply axiom $I3$ in $M$ to obtain a $x \in (I' \cup F') \setminus (I \cup F)$ such that $(I \cup F) + x \in \mathcal{L} (M)$. This means $I + x \in \mathcal{L} (M[k])$. And $x \in I' \setminus I$ follows, as $x \notin F'$ by our choice of $F$. To show axiom $I4$, let $I \subset X \subset E(M)$ with $I \in \mathcal{L} (M[k])$ be given. By the final axiom for $M$, there is a $B \in \mathcal{L} (M)$ which is maximal subject to $I \subset B \subset X$. We may assume that $F := B \setminus I$ has at most $k$ elements; for otherwise there is a superset $I' \subset B$ of $I$  such that $|B \setminus I'| = k$ and it suffices to find a maximal set containing $I' \in \mathcal{L} (M[k])$ instead of $I$.

	We claim that for any $F^{+} \subset X \setminus I$ of size $k+1$ the set $I \cup F^{+}$ is not in $\mathcal{L} (M[k])$. For a contradiction, suppose it is. Then in $M|X$, the set $B=I \cup F$ is a base and $I \cup F^{+}$ is independent and as $(I \cup F) \setminus (I \cup F) \setminus (I \cup F^{+}) \subset F \setminus F^{+}$ is finite, Lemma 3.3.1 implies

\begin{equation*}
|F^{+} \setminus F| = |(I \cup F^{+}) \setminus (I \cup F)| \leq |(I \cup F) \setminus (I \cup F^{+})| = |F \setminus F^{+}|.
\end{equation*}

This means $k+1=|F^{+}| \leq |F| = k$, a contradiction. So by successively adding single elements of $X \setminus I$ to $I$ as long as the obtained set is still in $\mathcal{L} (M[k])$, we arrive at the wanted maximal element after at most $k$ steps.
\end{proof}

	We now prove the lemma that was used in the above proof.
\begin{lemma}
(~\cite{Union} Lemma 4.14)
Let $M$ be a matroid and $I$, $B \in \mathcal{L} (M)$ with $B$ maximal and $B \setminus I$ finite. Then, $|I \setminus B| \leq |B \setminus I|$.
\end{lemma}

\begin{proof} The proof is by induction on $|B \setminus I|$. For $|B \setminus I| = 0$ we have $B \subset I$ and hence $B=I$ by maximality of $B$. Now suppose there is $y \in B \setminus I$. If $I+y \in \mathcal{L}$ then by induction

\begin{equation*}
|I \setminus B| = |(I+y) \setminus B| \leq |B \setminus (I+y)|= |B \setminus I| - 1
\end{equation*}

and hence $|I \setminus B| < |B \setminus I|$. Otherwise there exists a unique circuit $C$ of $M$ in $I+y$. Clearly $C$ cannot be contained in $B$ and therefore has an element $x \in I \setminus B$. Then $(I+y)-x$ is independent, so by induction

\begin{equation*}
|I \setminus B|-1 = |((I+y)-x \setminus B| \leq |B \setminus ((I+y)-x)| = |B \setminus I| - 1,
\end{equation*}

and hence $|I \setminus B| \leq |B \setminus I|$.
\end{proof}

\begin{proposition}
Consider the matroid $M_n[k]=(E_n, \mathcal{L}_n[k])$ where $M_n$ is the graph of $n$ parallel ladders from Section 3.2. We claim that $Spec(M_n[k])=\{ k,k+1,...,k+n \}$.
\end{proposition}

\begin{proof}
First notice that $M_n$ and $M_n[k]$ have the same finite independent subsets and so they have the same finitarization. Let $0 \leq j \leq n$. Suppose that $B_j$ is some base in $M_n$ that shows that $j \in Spec(M)$. Deleting $k$ elements from $B_j$ will give you a base $B_j[k]$ of $M_n[k]$. It is easy to see that $B_j[k]$ extends to some basis $B^\mathrm{fin}$ of $M_n^\mathrm{fin}$ by adding $k+j$ elements.
\end{proof}

A similar argument shows that $Spec(M_{\infty}[k]) = \mathbb{N}_{ \geq k} \cup \infty$.

\begin{question}
Does there exist a matroid $M$ and natural numbers $i < j <k$ such that $i,k \in Spec(M)$ and $j \notin Spec(M)$?
\end{question}

\begin{question}
Does there exist a matroid $M$ and natural numbers $i < j$ such that $i, \infty \in Spec(M)$ and $j \notin Spec(M)$?
\end{question}

\end{section}

\begin{section}{Nearly Finitary Matroids}
We now introduce a sufficient condition that forces a nearly finitary matroid to be $k$-nearly finitary.

\begin{theorem}
Let $M=(E,\mathcal{L})$ and $M^{\mathrm{fin}}=(E,\mathcal{L}^{\mathrm{\mathrm{fin}}})$ be a matroid and its finitarization respectively. Define $\hat{M} := (E,\mathcal{K})$ where $S\in\mathcal{K}$ if there exists a base $F$ in $M^{\mathrm{fin}}$ and a base $B$ in $M$ such that $B\subset F$ and $S\subset F\setminus B$. 

If $M$ is a nearly finitary matroid such that $\hat{M}$ is a matroid, then $M$ is $k$-nearly finitary.
\end{theorem}

\begin{proof}
Suppose $M=(E,\mathcal{L})$ is a matroid such that $\hat{M}=(E,\mathcal{K})$ is a matroid and $M$ is not $k$-nearly finitary. Since $\hat{M}$ is a matroid, it has a base. Suppose there is a base $B$ in $\hat{M}$ such that $|B|=k<\infty$. Since $M$ is not $k$-nearly finitary, there exists $S\in\mathcal{K}$ such that $|S|>k$. By the basis exchange property of matroids, all bases of a matroid either have the same finite cardinality or are all infinite in cardinality. Since $S$ is independent, it is a subset of a base $B_{S}$ with cardinality $k$. That is impossible. Thus, bases in $\hat{M}$ must have infinite cardinality. Since $\hat{M}$ has an independent set of infinite size, $M$ is not nearly finitary.
\end{proof}

For an example of a class of matroids where this structure is a matroid, suppose $M$ is a finitary matroid of infinite rank. Let $N := M[j]$. Then $\hat{N}$ is the matroid whose independent subsets are independent sets in $M$ with rank at most $j$.

Our $\hat{M}$ is suggestive of a difference structure. Let $N$ and $M$ be matroids on the same ground set with $\mathcal{L}(N) \supset \mathcal{L}(M)$. We define 
\begin{equation*}
\mathcal{L}(N \ominus M) := \{ S \colon \text{there exists } B_N \in \mathcal{L} (N)^\mathrm{max}, B_M \in \mathcal{L} (M)^\mathrm{max} \text{ with } B_N \supset B_M \text{ and } S \subset B_N \setminus B_M \}.
\end{equation*}

We then define $N \ominus M := (E(N), \mathcal{L}(N \ominus M))$. It is easy to see that $\hat{M}=M^\mathrm{fin} \ominus M$. 

\begin{theorem}
If $N$ and $M$ are matroids on finite ground sets satisfying the conditions for $N \ominus M$ to be defined, $N \ominus M = (N^* \vee  M)^*$.\footnote{We thank Nathan Bowler and Ann-Kathrin Elm from the University of Hamburg for suggesting this construction in private discussions.}
\end{theorem}

\begin{proof}
Let $B_M$ be a base of $M$. By assumption, $B_M$ is independent in $N$ and is thus contained in some base $B_N$ of $N$. So $B_N^*:=E(N) \setminus B_N$ is disjoint from $B_M$. So $B_N^* \cup B_M$ is a maximal independent subset of $N^* \vee M$. The complement of $B_N^* \cup B_M$ with respect to $E(N)$ is $B_N \setminus B_M$ and is a base of $(N^* \vee M)^*$. Conversely, bases of $N^* \vee M$ are of the form $A_N^* \cup A_M$ where $A_N^*$ and $A_M$ are bases of $N^*$ and $M$ respectively with $A_N^* \cap A_M = \emptyset$. We know it is possible to make $A_N^*$ and $A_M$ disjoint since $\mathcal{L}(M) \subset \mathcal{L}(N)$.  Then $A_N := E(N) \setminus A_N^*$ is a base of $N$ that contains the base $A_M$ of $M$. We then have
\begin{equation*}
E(N) \setminus (A_N^* \cup A_M) = A_N \setminus A_M.
\end{equation*}
We thus see that the bases of $N \ominus M$ are precisely the bases of $(N^* \vee M)^*$.
\end{proof}

We remark that the assumption that $E(N)=E(M)$ is a finite set is crucial. For infinite matroids, there exists examples of nearly finitary matroids $N$ and $M$ on a common ground set such that there are disjoint bases $B_N$ and $B_M$ of $N$ and $M$ respectively with $B_N \cup B_M$ non maximal in $N \vee M$. This kind of phenomenon was first observed in ~\cite{Union}.

We present an example different from the one in ~\cite{Union}. Suppose $M$ is the algebraic cycle matroid of the one way infinite ladder graph $G$ and suppose $N:=M^\mathrm{fin}$. As before, our one way infinite ladder starts from the left and extends infinitely far to the right. Then $N$ and $M$ satisfy the condition that allows us to define $N \ominus M=\hat{M}$. Since the $Spec(M)=\{0,1\}$, $\hat{M}$ has rank $1$. All rank $1$ independence systems are matroids since every non empty independent set is maximal. We can show that every singleton edge set is independent in $\hat{M}$. Consider the set $T$ of top edges and the set $S$ of bottom edges of $G$. Then $T \cup S$ is a base in $M$ but not a base in $M^\mathrm{fin}$. We get a base in $M^\mathrm{fin}$ by adding any middle edge. This shows that singleton sets of middle edges are independent in $\hat{M}$. Finally consider $B_N:=T \cup S + e$ where $e$ is the left most middle edge. $B_N$ is a base in $M^\mathrm{fin}$. Removing any single edge from $B_N$ gives us a base in $M$. Thus, all singleton sets of top or bottom edges are independent in $\hat{M}$. Thus, independent sets of $\hat{M}$ consist of the empty set and all singleton edge sets. Since $0 \in Spec(M)$, pick a base $B$ of $M$ that is also a base of $N$. With $B^*:= E(M) \setminus B$, we have that

\begin{equation*}
E(M)=B \cup B^* \in \mathcal{L}(N^* \vee M).
\end{equation*}

Thus, $(N^* \vee M)^*$ has only the empty set as its independent set. It is clear that $\hat{M} \neq (N* \vee M)^*$. Now consider $B_N$ from before. Define $B_{N^*}$ to be set of all middle edges excluding the left-most middle edge. Then $B_{N^*}$ is a base in $N^*$ disjoint from $B_M:=S \cup T$ which is a base in $M$. However, $B_{N^*} \cup B_M$ is a proper subset of $E(M)$ and $\mathcal{L}(N^* \vee M) = 2^{E(M)}$. Thus, a disjoint union of a base in $N^*$ and a base $M$ is not necessarily a base in $N^* \vee M$. $B_{N^*}$ being disjoint from $B_M$ means that $B_N=E(M) \setminus B(N^*)$ contains $B_M$. So $(N^* \vee M)^*$ can have bases that are strictly smaller than $\hat{M}$.

Another issue that arises is that $N^* \vee M$ is not always going to be a matroid. An example of two matroids whose union is not a matroid is shown in~\cite{Union}. By the Nearly Finitary Matroid Union theorem, $N^* \vee M$ is a matroid when $N^*$ and $M$ are both nearly finitary. It may still be the case that $(M^\mathrm{fin})^* \vee M$ is always a matroid but we have no proof of this as of now. This motivates the following theorem:

\begin{theorem}
Suppose $M$ is a nearly finitary and nearly cofinitary matroid. Then $(M^\mathrm{fin*} \vee M)^*$ is a matroid.
\end{theorem}

\begin{proof} Since we already have that $M$ is nearly finitary, it remains to show that $M^\mathrm{fin*}$ is nearly finitary. Let $F^*$ be a base in $M^\mathrm{fin*}$. Then $F:=E(M) \setminus F^*$  is a base in $M^\mathrm{fin}$ and contains some base $B$ of $M$. Since $M$ is nearly finitary, $|F \setminus B|$ is finite. Also $B^*:= E(M) \setminus B$ is a base of $M^*$ which contains $F^*$. Note that $|B^* \setminus F^*| = | F \setminus B| < \infty$. Since $M$ is nearly cofinitary, $B^*$ can be extended to some base in the finitarization of $M^*$ by adding finitely many elements. It follows that  we can extend a base $F^*$ of $M^\mathrm{fin*}$ to a base in the finitarization of $M^*$ by adding finitely many elements. Observe that bases in the finitarization of $M^\mathrm{fin*}$ are contained in bases of $(M^*)^\mathrm{fin}$. So we can extend $F^*$ to some base in $(M^\mathrm{fin*})^\mathrm{fin}$ by adding finitely many elements. This shows that $M^\mathrm{fin*}$ is nearly finitary. Now, we can apply the Nearly Finitary Matroid Union theorem to conclude that $M^\mathrm{fin*} \vee M$ is a matroid. Since the dual of a matroid is a matroid, the claim is proven.
\end{proof}

We propose the following independence system. Let $M$ be a nearly finitary matroid and define the following
\begin{equation*}
S(M):= \{ F^* \cup B \colon F^* \text{ is a base of } M^\mathrm{fin*}, B \text{ is a base of } M, \text{ and } F^* \cap B = \emptyset \}. 
\end{equation*}

We then define $S(M)_\mathrm{min}$ to be the minimal elements of $S(M)$ with respect to set inclusion. From before, bases of $(M^\mathrm{fin*}) \vee M$ are the maximal elements of $S(M)$ with respect to set inclusion. We speculate that $S(M)_\mathrm{min}$ is the set of bases of some matroid $A$ and that $A^* = \hat{M}$. As of now, we do not even have proof that $S(M)_\mathrm{min}$ is necessarily non-empty. However, we can prove the following.

\begin{theorem}
$S(M)_\mathrm{min}$ is non-empty if $M$ is a $k$-nearly finitary matroid.
\end{theorem}

\begin{proof} Suppose that $S_2:=F_{2}^* \cup B_2$ properly contains $S_1:=F_{1}^* \cup B_1$ with $F_{i}^*$ and $B_i$ being bases of $M^\mathrm{fin*}$ and $M$ respectively and $B_i \cap F_{i}^* = \emptyset$ for each $i \in \{ 1,2 \}$. Then $S_{1}^*:=E \setminus S_1$ properly contains $S_{2}^*:= E \setminus S_2$. $S_{i}^* = F_i \setminus B_i$ where $F_i= E \setminus F_{i}^*$ and $F_i \supset B_i$. Since $M$ is $k$-nearly finitary, $|S_{2}^*|<|S_{1}^*| \leq k$. We thus see that any chain of containments of elements of $S(M)$ must be finite. So $S(M)_\mathrm{min}$ is non-empty.

In her master's thesis ~\cite{Cofinitary}, Elm proved that every nearly finitary matroid $M$ that is also cofinitary must be $k$-nearly finitary for some $k$. From our previous theorem, this shows that every cofinitary nearly finitary matroid $M$ has a non-empty $S(M)_\mathrm{min}$.
\end{proof}

We have been considering matroids that we can construct so far. Let us now consider a hypothetical counterexample that would solve problem $P1$.
\begin{theorem}
Fix a $k \in \mathbb{N}$. Suppose that there exists some matroid $M=(E(M),\mathcal{L} (M))$ that is nearly finitary but not $n$-nearly finitary for all $n \in \mathbb{N}$. Then $M[k]$ is also nearly finitary but not $n$-nearly finitary for all $n \in \mathbb{N}$. 
\end{theorem}
\begin{proof}
First, we show that $M[k]$ is nearly finitary. Any base $B[k]$ of $M[k]$ can be extended to a base $B$ of $M$ by adding $k$ elements which can then be extended to a base of $M^{fin}=M[k]^{fin}$ by adding another finite set of elements. Since $M$ is not $n$-nearly finitary there is some base $F$ of $M^\mathrm{fin}$ containing some base $B$ of $M$ with $|F \setminus B| \geq n$. $B$ contains a base $B[k]$ in $M[k]$ and $|F \setminus B[k]| \geq n$.
\end{proof}

\begin{theorem}
Let $M=(E(M),\mathcal{L} (M))$ be a nearly finitary matroid that is not $k$-nearly finitary for any $k \in \mathbb{N}$ and let $N=(E(N), \mathcal{L}(N))$ be a nearly finitary matroid with $E(N)$ disjoint from $E(M)$. Then $M \vee N$ is a nearly finitary matroid that is not $k$-nearly finitary for any $k \in \mathbb{N}$.
\end{theorem}

\begin{proof}
By the nearly finitary matroid union theorem in ~\cite{Union}, the matroid union $M \vee N$ is a nearly finitary matroid. Moreover, we show that it is not $n$-nearly finitary. For any $n \in \mathbb{N}$ pick a base $F_M$ of $M^\mathrm{fin}$ and a base $B_M$ of $M$ such that $F_M$ contains $B_M$ and $|F_M \setminus B_M| > n$. We can do this since $M$ is not $n$-nearly finitary. Let $B_N$ be any base of $N$ and let $F_N$ be a base of $N^\mathrm{fin}$ containing $B_N$. Then $B_M \cup B_N$ is a base of $M \vee N$ contained in $F_M \cup F_N$. The difference between these two sets must at least contain $F_M \setminus B_M$ since $B_N$ is disjoint from $F_M$.  Moreover, $F_M \cup F_N$ is a base in $M^\mathrm{fin} \vee N^\mathrm{fin}$. By Proposition 4.12 in ~\cite{Union}, $M^\mathrm{fin} \vee N^\mathrm{fin} = (M \vee N)^\mathrm{fin}$ so $F_M \cup F_N$ is a base of $(M \vee N)^\mathrm{fin}$. So we have
\begin{equation*}
|F_M \cup F_N \setminus B_M \cup B_N| \geq |F_M \setminus B_M| > n.
\end{equation*}
Thus, $M \vee N$ is not $n$-nearly finitary as claimed earlier.
\end{proof}
We have thus showed that if there is even one counterexample to this conjecture, there is quite a large family of counterexamples. One natural question that arises is whether there is some refinement of our assumption that $E(N)$ is disjoint from $E(M)$ such that we can generate even more counterexamples.

Consider the finitary matroid $N=(\mathbb{N}, 2^{\mathbb{N}})$ with $\mathbb{N}$ disjoint from $E(M)$. Then $M \vee N$ is a nearly finitary matroid that is not $n$-nearly finitary. Moreover, $M \vee N$ has infinitely many coloops.

Note that if we completely drop the assumption that $E(N)$ is disjoint from $E(M)$, this statement is no longer true. Let $A_N=(E(A),\mathcal{L}(A_N))$ be a nearly finitary matroid that is not $k$-nearly finitary and let $A=(E(A), 2^{E(A)})$. By the nearly finitary matroid union theorem, $A[1] \vee A_N$ is a nearly finitary matroid. Since $A_N$ is not finitary, it has infinite rank and $E(A)$ must have infinite cardinality. Bases of $A_N$ have infinite cardinality and are thus non-empty. Let $e$ be an element of some base in $A_N$. Then $\{e\}$ is independent in $A_N$. $S_A:=E(A) - e$ is independent in $A[1]$ and $S_A + e=E(A)$ is independent in $A[1] \vee A_N$. So $A[1] \vee A_N=(E(A), 2^{E(A)})=A$. Consequently, $A[1]$ is exactly $1$-nearly finitary but $A[1] \vee A_N= A$ is finitary.

We can further extend this family using deletion and contraction operations defined earlier.
\begin{theorem}
Suppose that $M$ is a nearly finitary matroid that is not $n$-nearly finitary for any natural number $n$. Furthermore, suppose that $S \subset E(M)$ is a finite set. Then $M - S$ is a nearly finitary matroid that is not $n$-nearly finitary for any natural number $n$.
\end{theorem}
\begin{proof}
First, we show that $M^{fin} - S = (M-S)^{fin}$. Suppose that $I \in \mathcal{L} (M^{fin} -S)$. Then $I \in \mathcal{L} (M^\mathrm{fin})$ with $I \cap S = \emptyset$. Then all finite subsets of $I$ are independent in $M$ and also independent in $(M-S)$. So $I \in \mathcal{L} ((M-S)^\mathrm{fin})$. So $ \mathcal{L} (M^\mathrm{fin} -S) \subset \mathcal {L} ((M-S)^{fin})$. Conversely, suppose that $I \in \mathcal{L} (M-S)^\mathrm{fin}$. Then all finite subsets of $I$  are independent in $M-S$ and thus independent in $M$. So $I \in \mathcal{L}(M^\mathrm{fin})$.  Also, $I \cap S = \emptyset$. So $I \in \mathcal{L} (M^\mathrm{fin} -S)$ and we conclude that $\mathcal{L} (M^{fin} -S) = \mathcal {L} ((M-S)^\mathrm{fin})$.

Now $M-S$ is clearly nearly finitary since $M$ is. Since $M$ is not $n$-nearly finitary by assumption, suppose that $B \in M$, $F \in M^\mathrm{fin}$ such that $B \subset F$ and $ | F \setminus B| \geq n+2 \cdot |S|$. Then $F \setminus S$ can be extended to a base $F_{M-S}$ of $(M-S)^{fin}$ by adding at most $|S|$ elements. Similarly, $B \setminus S$ can be extended to a base $B_{M-S}$ of $(M-S)$ by adding at most $|S|$ elements from $F_{M-S}$. So we still have $B_{M-S} \subset F_{M-S}$.  It is clear that  $|F_{M-S} \setminus B_{M-S}| \geq n$. So $M-S$ is not $n$-nearly finitary.
\end{proof}

We still can further extend this result using special contractions.

From before, we used the nearly finitary matroid union theorem of ~\cite{Union} and a hypothetical existing example to construct examples of matroids that are nearly finitary but not $n$-nearly finitary. We constructed such an example with infinitely many coloops. We can use this example with coloops and contractions to get more examples.

\begin{theorem}
Suppose $M$ is a nearly finitary matroid that is not $n$-nearly finitary for any natural number $n$ and has some coloops. Let $T$ be a set of coloops of $M$. Then the contraction $M/T$ is a nearly finitary matroid that is not $n$-nearly finitary for any natural number $n$.
\end{theorem}

\begin{proof}
We first want to show that $M/T$ is nearly finitary. Since $T$ is a set of coloops, $M/T = (E \setminus T, \mathcal{L} (M) \cap 2^{E \setminus T})$. Also, $(M/T)^\mathrm{fin}=(M^{fin})/T$. Bases of $M/T$ are of the form $B \setminus T$ where $B$ is a base in $M$. All bases of $M$ contain $T$ and thus all bases of $M^{fin}$ also contain $T$. Since $M$ is nearly finitary, for all pairs $(F,B)$ such that $F \in M^{fin}$, $B \in M$ with $B \subset F$, $|F \setminus B| < \infty$. Since $F$ and $B$ both contain $T$,
\begin{equation*}
|(F \setminus T) \setminus (B \setminus T)| = |F \setminus B| < \infty.
\end{equation*}

Since $M$ is not $n$-nearly finitary, there exists a base $B_n$ of $M$ contained in a base $F_n$ of $M^\mathrm{fin}$ such that $|F_n \setminus B_n| > n$. Again, $F_n$ and $B_n$ both contain $T$. So
\begin{equation*}
|(F_n \setminus T) \setminus (B_n \setminus T)| = |F_n \setminus B_n| > n.
\end{equation*}

So $M/T$ is not $n$-nearly finitary.
\end{proof}

The $[k]$ map in ~\cite{Union} was constructed to generate $k$-nearly finitary matroids from infinite rank finitary matroids. Earlier, we defined what it means for a matroid $M$ to be exactly $k$-nearly finitary. Recall that $M$ is exactly $k$-nearly finitary if it is $k$-nearly finitary but not $(k-1)$-nearly finitary. Here, we discuss a small refinement of what the authors of ~\cite{Union} meant. It is rather clear that if $M$ is a finitary matroid of infinite rank, then $M[k]$ is exactly $k$-nearly finitary. Also, if $M$ is exactly $j$-nearly finitary and of infinite rank, then $M[k]$ is exactly $(k+j)$-nearly finitary.

Suppose that $A=(E(A), \mathcal{A})$ is exactly $k$-nearly finitary and $B=(E(B), \mathcal{B})$ is exactly $j$-nearly finitary with $E(A)$ disjoint from $E(B)$. Then the matroid union $A \cup B$ is exactly $k+j$-nearly finitary. This $[k]$ map motivates us to consider possible $[-j]$ maps where $j$ is a natural number. Let $m= \max \{ 0, n+k-j \}$. Let $M$ be exactly $n$-nearly finitary and of infinite rank.

\begin{question}
Does there exist a $[-j ]$ map s.t. $((M[k])[-j])$ and $((M[-j])[k])$ are both exactly $m$-nearly finitary for all $k \in \mathbb{N}$?
\end{question}

\end{section}

\begin{section}{Unionable Matroids}
The Nearly Finitary Matroid Union theorem in ~\cite{Union} gives us a matroid union theorem for a pair of nearly finitary matroids. This motivates us to consider other pairs of matroids. We will show that the union of a finite rank matroid and an arbitrary matroid is again a matroid. We define a matroid $M$ to be $\textit{unionable}$ if $M \vee N$ is a matroid for any given matroid $N$.

\begin{theorem}
Finite rank matroids are unionable.
\end{theorem}

\begin{proof}
Let $F$ be a finite rank matroid and let $N$ be any matroid. From ~\cite{Union}, $F\vee N$ satisfies axioms $I1$ through $I3$. For axiom $I4$, we will use proof by contradiction. Let $I\subset X\subset E_{F}\cup E_{N}$ and suppose $I$ is an independent set in $F \vee N$. By definition, there exists $I_{F}\in\mathcal{L}_{F}$ and $I_{N}\in\mathcal{L}_{N}$ such that $I=I_{F}\cup I_{N}$. Then 
\begin{equation*}
I_{N}\subset X\cap E_{N}\subset E_{N}.
\end{equation*}
Since $N$ is a matroid, the set 
\begin{equation*}
\{I_{N}':I_{N}\subset I_{N}'\subset X\cap E_{N}\}
\end{equation*}
has some maximal element $B_{N}$. Clearly, $B_{N}$ is independent in $F\vee N$ and is an element of the set 
\begin{equation*}
S:=\{I'\in\mathcal{L}_{F\vee N}\colon I\subset I'\subset X\}.
\end{equation*}
Suppose there exists some independent set $J$ of $F\vee N$ such that $B_{N}\subset J\subset X$ and $|J\setminus B_{N}|>k$. We know that $J=J_{F}\cup J_{N}$ for some $J_{F}\in\mathcal{L}_{F}$ and $J_{N}\in\mathcal{L}_{N}$. Then $J\setminus J_{N}\subset J_{F}$ and we get $|J\setminus J_{N}|\leq|J_{F}|\leq k$. Since $N$ is a matroid, the restriction $N|J$ is also a matroid. By construction, $B_{N}$ is a base of $N|J$. $J_{N}$ is contained in some base $B_{J}$ of $N|J$. It is clear that $|J\setminus B_{J}|\leq k<|J\setminus B_{N}|$. By base exchange property, for all $b_{n}\in B_{N}\setminus B_{J}$, there is $b_{j}\in B_{J}\setminus B_{N}$ such that $B_{1}:=B_{J}+b_{n}-b_{j}\in\mathcal{L}_{N}$. Clearly, $|B_{N}\setminus B_{J}|\leq|J\setminus B_{J}|\leq k$. Iterating this process finitely many times will give us some base $B$ such that $|J\setminus B|=|J\setminus B_{J}|<|J\setminus B_{N}|$ and $B_{N}\subset B$. This shows that $B_{N}$ is a proper subset of the base $B$ of $N|J$. Contradiction.

We thus can find a maximal element of $S$ by adding at most $k$ elements to $B_{N}$.
\end{proof}
\end{section}

   \chapter[%
      Short Title of 4th Ch.
   ]{%
     Near Finitarization
   }%
   \label{ch:4thChapterLabel}
   \begin{section}{Near Finitarization}
Let $M=(E, \mathcal{L})$ and $M^\mathrm{{fin}}=(E, \mathcal{L}^{\mathrm{fin}})$ be a matroid and its finitarization respectively. Define
\begin{equation*}
\mathcal{L}^\mathrm{nfin} := \{ F \in \mathcal{L}^\mathrm{fin} \colon \exists S \in \mathcal{L} \text{ s.t. } S \subset F, |F \setminus S| < \infty  \}.
\end{equation*}
We define $M^{\mathrm{nfin}}:=(E,\mathcal{L}^{\mathrm{nfin}})$. We call this the near finitarization of $M$. 

\begin{theorem}
$M^{\mathrm{nfin}}$  is a finitary matroid if and only if $M$ is nearly finitary.
\end{theorem}

\begin{proof} Suppose $M$ is a nearly finitary matroid. Pick $F \in \mathcal{L}^\mathrm{fin}$. Then there exists a base $B_{F}$ in $M^\mathrm{fin}$ with $F \subset B_{F}$. Since $M$ is nearly finitary, there exists a base $B$ in $M$ with $B \subset B_{F}$ and $|B_{F} \setminus B| < \infty$. Then $F \setminus (B_{F} \setminus B)$ is a subset of $F$ in $\mathcal{L}$ that removes finitely many elements of $F$.  Thus, $F \in \mathcal{L}^\mathrm{nfin}$. So $\mathcal{L}^\mathrm{fin} \subset \mathcal{L}^\mathrm{nfin}$. By construction, $\mathcal{L}^\mathrm{nfin} \subset \mathcal{L}^\mathrm{fin}$. Thus, $M^{\mathrm{nfin}} = M^{\mathrm{fin}}$. So $M^{\mathrm{nfin}}$ is a finitary matroid. Conversely, suppose $M$ is a matroid that is not nearly finitary. Then there exists a base $B$ in $M$ and a base $B_{F}$ in $M^{\mathrm{fin}}$  such that $B \subset B_{F}$ and $|B_{F} \setminus B| = \infty$. Then the set $\{S \in \mathcal{L}^{\mathrm{nfin}} : B \subset S \subset B_{F}\}$  has no maximal element. This violates the fourth matroid axiom so $M^\mathrm{nfin}$ is not a matroid.
\end{proof}

\begin{conjecture}
If $M$ is not $k$-nearly finitary, then $M^\mathrm{nfin}$ is not a matroid.
\end{conjecture}

\begin{theorem}
The above conjecture is equivalent to the conjecture that every nearly finitary matroid is $k$-nearly finitary.
\end{theorem}

\begin{proof} Let $M$ be a matroid. Suppose every nearly finitary matroid is $k$-nearly finitary for some $k \in \mathbb{N}$. If $M$ is not $k$-nearly finitary for all $k \in \mathbb{N}$, then $M$ is not nearly finitary and so $M^\mathrm{nfin}$ is not a matroid by Theorem 4.1.1. Conversely, suppose that $N^\mathrm{nfin}$ is not a matroid whenever $N$ is a matroid that is not $k$-nearly finitary for all $k \in \mathbb{N}$. Suppose $M$ is not $k$-nearly finitary for all $k \in \mathbb{N}$. Then $M^\mathrm{nfin}$ is not a matroid by assumption. Applying Theorem 4.1.1 shows that $M$ is not nearly finitary.
\end{proof}

Conjecture 4.1.1 may give us a new way to study problem $P1$.

In ~\cite{Dray}, the authors prove a longstanding conjecture by Andreae on graphs. In ~\cite{Halindray}, Halin himself extended his theorem in the following way. If, for all $k \in \mathbb{N}$, a graph $G$ has a set of $k$ vertex-disjoint double rays, then it has some set of infinitely many vertex-disjoint double rays. Andreae extended Halin's theorem for edge disjoint rays and further conjectured that an analogous result would hold for edge-disjoint double rays.

The motivation for studying problem $P1$ comes from trying to generalize Halin's theorem. Here, we explore some other possible ways to generalize Halin's theorem.

Now double rays are infinite circuits of algebraic cycle matroids. Thus, we further conjecture the following. 
\begin{conjecture}
If, for all $k \in \mathbb{N}$, a matroid $M$ has some set of $k$ disjoint infinite circuits, then $M$ has some set of infinitely many disjoint infinite circuits.
\end{conjecture}
The ground set of an algebraic cycle matroid of a graph $G$ is the edge set of $G$. So this conjecture extends the edge-disjoint double ray result recently found. 

We also present the following conjecture related to Conjecture 4.1.2.
\begin{conjecture}
If, for all $k \in \mathbb{N}$, a matroid $M$ has some set of $k$ disjoint infinite circuits whose union is independent in $M^\mathrm{fin}$, then it has some set of infinitely many disjoint infinite circuits whose union is independent in $M^\mathrm{fin}$.
\end{conjecture}
Now any union of vertex disjoint double rays of a graph $G$ is independent in the finitarization of the algebraic cycle matroid of $G$. So this conjecture extends the vertex-disjoint double ray result found in 1981 by Andreae.
\end{section}

\begin{section}{Independence System Example}
In the introduction, we showed that the notion of finitarization can also be extended to independence systems that are not matroids. We also defined what it means for an independence system to be nearly finitary and $k$-nearly finitary. We will now show by an explicit construction that not every nearly finitary independence system is $k$-nearly finitary.
\begin{theorem}
There exists a nearly finitary independence system that is not $k$-nearly finitary.
\end{theorem}
\begin{proof}
 Consider the set
\begin{equation*}
\mathcal{T} :=\{\{1\},\{2,3\},\{4,5,6\},...\}.
\end{equation*}
 Consider the following:
\begin{equation*}
\mathcal{L} := \{ S \colon S \subset \mathbb{N} \setminus T \text{ with } T \in \mathcal{T} \}.
\end{equation*}
Then $N:=(\mathbb{N}, \mathcal{L})$ gives us our desired independence system. It is clear that the empty set is independent and that subsets of independent sets are independent. It is nearly finitary since $|F \setminus B| < \infty$ whenever $F$ is a base of $N^\mathrm{fin}$ that contains $B$ and $B$ is a base of $N$. In fact, the only base of $N^\mathrm{fin}$ is $\mathbb{N}$. Every base of $N$ is missing only finitely many elements from $\mathbb{N}$. However, there is no finite bound on how many elements a base of $N$ can miss so $N$ is not $k$-nearly finitary for any $k \in \mathbb{N}$.
\end{proof}
However, $N$ is not a matroid. To see why, consider $S_1 := \mathbb{N} \setminus \{ 1, 2 \}$ and $B := \mathbb{N} \setminus \{ 2 ,3 \}$. $B$ is a maximal independent set in $N$ and $S_1$ is not.  $2$ is not in $B$ and $S_1 + 1$ is not independent. Thus, axiom $I3$ is not satisfied. However, every finite subset of $\mathbb{N}$ is independent so it is clear that $N$ satisfies axiom $F3$. $N$ also satisfies axiom $I4$. To see why, let $X$ be some subset of $\mathbb{N}$. Let $I$ be an independent set of $N$ contained within $X$. Since $I$ is independent, there is some set $T \in \mathcal{T}$ such that $I \subset \mathbb{N} \setminus T$. If $X$ is independent in $N$, then $X$ is the largest subset of $X$ that is independent in $N$. If $X$ is not independent in $N$, then $X \setminus T$ is a subset of $X$ independent in $N$ and containing $I$. Since $T$ is finite, there must be some maximal independent subset of $X$ which contains $X \setminus T$. So $N$ seems to almost be a matroid.

Since all finite subsets of $\mathbb{N}$ are independent in $N$, $N^{\mathrm{fin}}=(\mathbb{N},2^{\mathbb{N}})$. The only base of $N^{\mathrm{fin}}$ is $\mathbb{N}$. A set $S \subset \mathbb{N}$ is a base of $N$ if and only if there is some $T \in \mathcal{T}$ with $S= \mathbb{N} \setminus T$. Let $S_0:=\mathbb{N} \setminus T_0$ be some base in $N$ with $T_0 \in \mathcal{T}$. Since $|T|< \infty$ for all $T \in \mathcal{T}$, $|\mathbb{N} \setminus S_0| = |T_0|<\infty$. So $N$ is a nearly finitary independence system. However, for all $k \in \mathbb{N}$, there is some $T_k \in \mathcal{T}$ with $|T_k| > k$. So we can get a base $S_k := \mathbb{N} \setminus T_k$ with $|\mathbb{N} \setminus S_k| = |T_k| > k$. So $N$ is not a $k$-nearly finitary independence system with our definitions. We have thus shown that not every nearly finitary independence system is $k$-nearly finitary. We remark that $N$ is $1$-nearly finitary under the definition of $1$-nearly finitary in ~\cite{Union}. This alternate definition can be found on page 6 of our introduction. So although our definitions of nearly finitary and $k$-nearly finitary agree with the definitions in ~\cite{Union} for matroids, they do not agree for more general independence systems. We use our definitions of nearly finitary and $k$-nearly finitary over the ones in ~\cite{Union} because this choice makes our example independence system $N$ interesting.

The previous example already shows that our definition of $k$-nearly finitary is different from the definition in ~\cite{Union}. Our next example will show that our definition of nearly finitary is different from the one in ~\cite{Union}. Consider the following set:
\begin{equation*}
\mathcal{R} := \{\mathbb{N}, \{-1\}, \{-2\}, \{-3\}, ... \}.
\end{equation*}
 Consider the following:
\begin{equation*}
\mathcal{L}(\mathcal{R}) := \{ S \colon S \subset \mathbb{R} \setminus R \text{ with } R \in \mathcal{R} \}.
\end{equation*}
Then $I:=(\mathbb{R}, \mathcal{L}(\mathcal{R}))$ is an independence system. Every finite subset of $\mathbb{R}$ is independent in $I$ so $I^{\mathrm{fin}}=(\mathbb{R},2^{\mathbb{R}})$. $I$ satisfies the definition of $1$-nearly finitary in ~\cite{Union} since the only base $F:=\mathbb{R}$ of $I^{\mathrm{fin}}$ contains a base $B:=\mathbb{R}\setminus \{-1\} $ in $I$ such that $|F \setminus B|=|\{ -1 \}|=1$. So $I$ is nearly finitary in the sense of ~\cite{Union}. However, $B_0:=\mathbb{R} \setminus \mathbb{N}$ is a base in $I$ with $\mathbb{R} \setminus B_0 =\mathbb{N}$. Since $\mathbb{N}$ does not have finite cardinality, $I$ is not nearly finitary in our sense.

Finally, this does not contradict the matroid equivalence of our definitions of nearly finitary and $k$-nearly finitary with the definitions in ~\cite{Union}. To see why $I$ is not a matroid consider the independent set $S:= \mathbb{R} \setminus \{-1,1\}$ and the base $B_1 := \mathbb{R} \setminus \mathbb{N}$. No element of $B_1 \setminus S$ can be added to $S$ to get a new independent set in $I$. So $I$ does not satisfy axiom $I3$ and is not a matroid.
\end{section}

   \chapter[%
      Short Title of 5th Ch.
   ]{%
     $\Psi$-Matroids and Related Constructions
   }%
   \label{ch:5thChapterLabel}
   \begin{section}{$\Psi$-Matroids}
Bowler and Carmesin studied $\Psi$-matroids in ~\cite{Games}. Every graph $G$ has an End compactification $|G|$ obtained by adding compactification points for each end of a graph. We will define $|G|$ later. We define a graph $G$ to be locally finite if each vertex in $G$ has finitely many adjacent edges. When $G$ is an infinite, locally finite, and connected graph, $|G|$ can be given a compact topology. So this End compactification can often give us compactifications in the topological sense. Let $\Omega (G)$ be the set of ends of a graph. Let $\Psi$ be a subset of $\Omega (G)$. Then you can partially compactify $G$ by adding compactification points associated to each end in $\Psi$. This new structure will be denoted by $G_{\Psi}$.

We review the End compactification of a graph as shown in ~\cite{Games}. Let $G$ be a graph and let $\Omega (G)$ be its set of ends. Let $\sqcup$ be the disjoint union binary operation on sets. Let $d$ be the distance function on $V(G) \sqcup (0,1) \times E(G)$. Here, $V(G) \sqcup (0,1) \times E(G)$ is the ground set of the simplicial 1-complex formed from the vertices and edges of $G$. We define a topology $V_t$ on the set $V(G) \sqcup \Omega(G) \sqcup (0,1) \times E(G)$ by taking basic open neighbourhoods as follows:
\begin{itemize}
\item 1: For $v \in  V(G)$, the basic open neighborhoods of $v$ are $\epsilon$-balls $B_{\epsilon} (v) =\{ x|d(v,x)< \epsilon \}$ for $ \epsilon \leq 1$.
\item 2: For $(x, e) \in (0,1) \times E$ we say $(x,e)$ is an interior point of $e$, and take the basic open neighborhoods to be $\epsilon$-balls about $(x,e)$ with $ \epsilon \leq \min \{ x,1-x \}$.
\item 3: For $ \omega \in \Omega (G)$, the basic open neighborhoods of $\omega$ will be parametrized by the finite subsets $S$ of $V(G)$. Given such a subset, we let $C(S, \omega)$ be the unique component of $G \setminus S$ that contains a ray from $\omega$, and let $\hat{C} (S, \omega)$ be the set of all vertices and inner points of edges contained in or incident with $C(S, \omega)$, and of all ends represented by a ray in $C(S, \omega)$. We take basic open neighborhoods of $\omega$ to be the sets $\hat{C} (S, \omega)$.
\end{itemize}

The topological space obtained in this way is $|G|$ mentioned earlier.

For any set $\Psi$ of ends in $G$, we set $\Psi^C = \Omega (G) \setminus \Psi$ and $|G|_{\Psi}=|G| \setminus \Psi^{C}$.

Bowler and Carmesin show that the topological space $|G|_{\Psi}$ derived from the graph $G$ almost fits the notion of a graph-like space that appears in ~\cite{graphlike} in the following sense.
\begin{definition}
(~\cite{Games} Definition 3.2)
An almost graph-like space $G$ is a topological space (also denoted $G$) together with a vertex set $V=V(G)$, an edge set $E=E(G)$ and for each $e \in E$ a continuous map $\ell_e \colon [0,1] \rightarrow G$ such that:
\begin{itemize}
\item 1: The underlying set of $G$ is $V \sqcup (0,1) \times E$.
\item 2: For any $x \in (0,1)$ we have $\ell_e (x)=(x,e)$.
\item 3: $\ell_e (0)$ and $\ell_e(1)$ are vertices (called the endvertices of $e$).
\item 4: $\ell_e |_{(0,1)}$ is an open map.
\end{itemize}
\end{definition}

They then give $|G|_{\Psi}$ the structure of an almost graph-like space, with edge set $E(G)$ and vertex set $V(G) \cup \Psi$. By definition, such an almost graph-like space is a $\textit{graph-like space}$ if in addition for any $v, v' \in V$ with $v \neq v'$, there are disjoint open subsets $U$, $U'$ of $G$ partitioning $V(G)$ and with $v \in U$ and $v' \in U'$. This ensures that $V(G)$, considered as a subspace of $G$, is totally disconnected, and that $G$ is Hausdorff. Bowler and Carmesin go on to define an equivalence relation $\simeq$ such that $|G|_{\Psi} / \simeq$ defines a graph like space. Denote $|G|_{\Psi} / \simeq$ as $\tilde{G}_{\Psi}$. They then consider topological circles of $\tilde{G}_{\Psi}$ as circuits of a $\Psi$-system $M(G,\Psi)=(E(G), \mathcal{L})$ where independent sets of $\mathcal{L}$ are subsets of $E(G)$ that do not contain any of these topological circles from $\tilde{G}_{\Psi}$.  Then they consider the question about when this $M(G, \Psi)$ is a matroid. When $M(G, \Psi)$ is a matroid, we say that it is a $\Psi$-matroid. Bowler and Carmesin prove that if $G$ is a locally finite graph, and $\Psi$ a Borel set of ends of $G$, then $M(G, \Psi)$ is a matroid. Recall that a Borel set of a topological space is any set that can be formed from open sets through the operations of countable union, countable intersection, and relative complement.

\begin{definition}
$v \simeq v'$ if for every two disjoint open subsets $U$, $U'$ of $G$ partitioning $V(G)$, either $v,v' \in U$ or $v,v' \in U'$. In other words, the two vertices are always in the same part of this kind of bipartition of $V(G)$.
\end{definition}

Under this definition, two distinct vertices of a finite graph $G$ are not equivalent. Consider the topological space $|G|$ and let $v$ and $v'$ be two distinct vertices of $|G|$. Then the sets $B_{0.5} (v)$ and $\bigcup_{x \in V(G) \setminus {v}} B_{0.5} (x)$ form a bipartition of $V(G)$ and are two disjoint open subsets of $|G|$ separating $v$ and $v'$.

To see an example of a graph-like space $|G|$ and a pair $(v,v')$ of distinct vertices such that $v \simeq v'$, consider the Bean graph as shown on page \pageref{BeanGraph} of this thesis. For the reader's convenience, we also include a picture of the Bean graph in Figure 5.1. See ~\cite{InfGraph} to learn more about the Bean graph.

\begin{figure}
\includegraphics{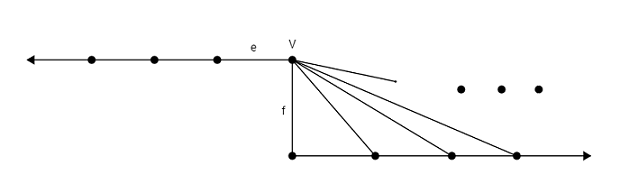}
\caption{Bean Graph}
\label{BeanGraph2}
\end{figure}

Here, the vertex $V$ 'dominates' the end $\omega$ on the right. The following definition is from ~\cite{Games}.

\begin{definition}
We say a vertex $v$ dominates a ray $R$ if there are infinitely many paths from $v$ to $R$. We say a vertex $v$ dominates an end $\omega$ if it dominates some ray (or equivalently, all rays) belonging to $\omega$.
\end{definition}

Bowler and Carmesin have shown that if $v$ dominates some end $\omega$, then $v$ and $\omega$ are equivalent as vertices of $|G|_{\Psi}$. To see why this is true, suppose that we bipartition the vertices of $|G|_{\Psi}$ with two disjoint open sets. Any open set that contains $\omega$ contains all but finitely many of the infinitely many paths from $v$ to a ray $R$ of $\omega$. Thus, an open set containing $\omega$ must also contain $v$. Following Definition 5.1.2, $\omega$ and $v$ must be equivalent.

\begin{lemma}
Let $M_{1}=(E,\mathcal{L}_{1})$ and $M_{2}=(E,\mathcal{L}_{2})$ be matroids on a common ground set. If $\mathcal{L}_{1}\subset\mathcal{L}_{2}$, then every base $B_{1}$ of $M_{1}$ extends to a base $B_{2}$ of $M_{2}$. Moreover, every base $B_{2}$ of $M_{2}$ contains a base $B_{1}$ of $M_{1}$.
\end{lemma}

\begin{proof}
Suppose that $B_{1}$ is a base in $M_{1}$. Then it is independent in $\mathcal{L}_{2}$ and so it extends to some base in $M_{2}$. Conversely, suppose that $B_{2}$ is a base in $M_{2}$. Observe that $\mathcal{L}_{1}^{*}\supset\mathcal{L}_{2}^{*}$. Now $E\setminus B_{2}$ is a base of $M_{2}^{*}$. Since it is independent in $\mathcal{L}_{1}^{*} $, it extends to some base $B^{*}$ in $M_{1}^{*} $. Since $E\setminus B_{2}\subset B^{*}$ , $B_{2}$ contains the complement of $B^{*}$ which is a base in $M_{1}$.
\end{proof}

\begin{lemma}
(~\cite{Inter} Proposition 1.4)
The algebraic cycle matroid $M_{AC}(G)$ is nearly finitary if and only if $G$ has a finite number of vertex disjoint rays.
\end{lemma}

\begin{theorem}
Suppose $M=(E,\mathcal{L})$ is a $\Psi$-matroid on a graph $G$ with finitely many vertex disjoint rays. Then $M$ is $k$-nearly finitary.
\end{theorem}

\begin{proof}
Consider the algebraic cycle matroid $M_{AC}=(E,\mathcal{L}_{AC})$. Suppose you have an independent set $S$ in $M_{AC}$. Then $S$ contains no finite cycle of $M$ since it contains no algebraic cycle. It also contains no double ray. An infinite circuit $C$ of $M$ is a non-empty set of double rays between elements of $\Psi$ such that $C$ is homeomorphic to the unit circle. For example, you can connect an end to itself with a double ray. So $S$ cannot contain an infinite circuit of $M$. Thus, $S$ has no circuit of $M$ and must be independent in $M$. So $\mathcal{L}_{AC}\subset\mathcal{L}$. $M_{AC}^\mathrm{fin}$ and $M^\mathrm{fin}$ are the finite cycle matroids of $G$ so they are the same. Suppose you have a base $B$ in $M$. By our Lemma 5.1.1, it contains some base $B_{AC}$ in $M_{AC}$. Since $M_{AC}$ is $k$-nearly finitary, you can add a set $A$ of $k$ or less elements to $B_{AC}$ to get a base in $M_{AC}^\mathrm{fin}=M^\mathrm{fin}$. Adding an appropriate subset of $A$ to $B$ will give that same base. So $M$ is $k$-nearly finitary.
\end{proof}
\end{section}

\begin{section}{$P( \Psi)$-Matroids}
Let us now consider a new but related construction. Recall that a $\Psi$-system of a graph $G$ picks a subset $\Psi$ of $\Omega$ and constructs a corresponding graph-like space. We want to consider other possible graph-like spaces associated with $G$.

Let $G$ be a fixed infinite, locally finite, and connected graph. Recall that $|G|$ is a compactification of $G$ which adds ends to $G$. Here, we define a different class of compactifications of $G$ using quotients of $|G|$. This will allow us to construct independence systems that are different from the $\Psi$-systems from above. Let $P_{\Omega}$ be a partition of $\Omega$. Consider the partition $P_{|G|}$ of $|G|$ defined as follows. $P_{|G|}$ is the partition of singletons on $|G|\setminus\Omega$. Then $P_{|G|}$ induces an equivalence relation $\sim_{P}$ on $|G|$. We consider the quotient space $|G|/\sim_{P}$. From this $|G|/\sim_{P}$, we can construct graph-like spaces. Then from these graph-like spaces, we can construct independence systems similar to $\Psi$-systems by picking a subset of $P_{\Omega}$.

It is easy to see that $|G|/\sim_{P}$ is compact. By Theorem 1.1 in ~\cite{Diestel}, $|G|$ is compact and so is any quotient space of $|G|$. We thus have a class of compactifications that interpolates between one point compactification and Freudenthal compactification.

For arbitrary graphs $G$, we can still add points in this way. Without our assumptions of local finiteness and connectedness, it is no longer clear that $|G|/\sim_{P}$ is compact.

From an arbitrary graph $G$, we have constructed spaces $|G|/\sim_{P}$. Let us name this the class of nearly-$P( \Psi )$ graphs. When $P_{|G|}$  agrees with $P_{\Omega}$ on $\Omega$ and no end $\omega \in \Omega (G)$ is identified with a point in the topological subspace $G \subset |G|$, we say it is a $P (\Psi)$-graph. We can follow the same construction as above and define an almost graph-like space with edge set $E(G)$ and vertex set $V(G) \cup P_{\Psi}$.

Let $\Psi$ be a subset of $\Omega(G)$ formed by the union of some subset of elements of $P( \Omega)$. We let $P( \Psi )$ be the partition of $\Psi$ given by the restriction of $P( \Omega)$ to $\Psi$.  Set $|G|_{ P( \Psi )}=(|G|_{\Psi}/\sim_{P})$. Then $|G|_{P( \Psi)}$ is also an almost graph-like space. We can get a graph-like space by taking $|G|_{P( \Psi)}/\simeq$.

The spaces we have constructed are special cases of the one that appear in ~\cite{graphlike}. In  ~\cite{InfGraphic}, Bowler et al. study matroids arising from graph-like spaces. We thus see our constructions as intermediate classes of spaces in the following sense:

\begin{equation*}
\Psi  \subset P( \Psi ) \subset \text{ nearly-}P( \Psi) \subset \text{ graph like} \subset \text { almost graph like}.
\end{equation*}

So our classes seem to be a middle ground in terms of difficulty of study. For these graphs we constructed, we can consider topological circles of nearly-$P( \Psi )$ graphs and ask when these circles give the circuits of a matroid. Regardless of whether they give a matroid, these topological circles will always give us circuits of some independence system. If the topological circles of a nearly-$P ( \Psi )$-graph forms a matroid, we say that we have a $\textit{nearly-}P (\Psi) \textit{ matroid}$. We use similar terminology for $P( \Psi)$-graphs and $P( \Psi )$-matroids. We can extend our earlier theorem to these new constructions.

\begin{theorem}
Suppose $M=(E, \mathcal{L})$ is a nearly-$P( \Psi )$-matroid on a graph $G$ with finitely many disjoint rays. Then $M$ is $k$-nearly finitary.
\end{theorem}

\begin{proof}
It is just like the proof for $\Psi$ matroids. Consider the algebraic cycle matroid $M_{AC}=(E,\mathcal{L}_{AC})$. Suppose you have an independent set $S$ in $M_{AC}$. Then $S$ contains no finite cycle of $M$ since it contains no algebraic cycle. It also contains no double ray. An infinite circuit $C$ of $M$ is a non-empty set of double rays between elements of $P( \Psi )$ such that $C$ is homeomorphic to the unit circle. So $S$ cannot contain an infinite circuit of $M$. Thus, $S$ has no circuit of $M$ and must be independent in $M$. So $\mathcal{L}_{AC}\subset\mathcal{L}$. $M_{AC}^\mathrm{fin}$ and $M^\mathrm{fin}$ are the finite cycle matroids of $G$ so they are the same. Suppose you have a base $B$ in $M$. By our Lemma 5.1.1, it contains some base $B_{AC}$ in $M_{AC}$. Since $M_{AC}$ is $k$-nearly finitary, you can add a set $A$ of $k$ or less elements to $B_{AC}$ to get a base in $M_{AC}^\mathrm{fin}=M^\mathrm{fin}$. Adding an appropriate subset of $A$ to $B$ will give that same base. So $M$ is $k$-nearly finitary.
\end{proof}

Our results are small steps towards a larger question.

Question: When are nearly $P( \Psi)$ matroids nearly finitary and when are they $k$-nearly finitary.

Another interesting direction to pursue is to ask when our nearly $P( \Psi )$ constructions give matroids. In the comment section of ~\cite{Blog}, Bowler asks the question of when a $P( \Psi )$-independence system for a graph $G$ gives a matroid in the case where $P( \Psi)$ is a bipartition of $\Omega(G)$. It does not seem to be clear that a bipartition of two Borel sets will give a matroid.
\end{section}

\begin{section}{Matroids and Axiom Systems}

Bowler and Carmesin also introduce a different way to construct matroids in the following sense:

Here, we allow trees to be infinite.

\begin{definition}
(~\cite{Games} Definition 5.1)
A tree $\mathcal{T}$ of matroids consists of a tree $T$, together with a function $M$ assigning to each node $t$ of $T$ a matroid $M(t)$ on a ground set $E(t)$, such that for any two nodes $t$ and $t'$ of $T$, if $E(t) \cap E(t')$ is nonempty then $tt'$ is an edge of $T$.
For any edge $tt'$ of $T$, we set $E(tt')=E(t) \cap E(t')$. We also define the ground set of $\mathcal{T}$ to be $E=E( \mathcal{T}) = ( \bigcup_{t \in V(T)} E(t)) \setminus ( \bigcup_{tt' \in E(T)} E(tt'))$.
We shall refer to the edges which appear in some $E(t)$ but not in $E$ as dummy edges of $M(t)$: thus the set of such dummy edges is $\bigcup_{tt' \in E(t)} E(tt')$.
\end{definition}

\begin{definition}
A tree $\mathcal{T} = (T,M)$ of matroids is of overlap $1$ if, for every edge $tt'$ of $T$, $|E(tt')|=1$. In this case, we denote the unique element of $E(tt')$ by $e(tt')$.
\end{definition}

Given a tree of matroids of overlap $1$ as above and a set $\Psi$ of ends of $T$, Bowler and Carmesin go on to define $\Psi$-pre-circuits and $\Psi$-circuits of $\mathcal{T}$. A $\Psi$-pre-circuit  of $\mathcal{T}$ consists of a subtree $C$ of $T$ together with a function $o$ assigning to each vertex $t$ of $C$ a circuit of $M(t)$, such that all ends of $C$ are in $\Psi$ and for any vertex $t$ of $C$ and any vertex $t'$ adjacent to $t$ in $T$, $e(tt') \in o(t)$ if and only if $t' \in C$. The set of $\Psi$-pre-circuits is denoted $\overline{\mathcal{C}} ( \mathcal{T}, \Psi)$. Any $\Psi$-pre-circuit $(C,o)$ has an underlying set $(C,o)=E \cap \bigcup_{t \in V(C)} o(t)$. Nonempty subsets of $E$ arising in this way are called $\Psi$-circuits of $\mathcal{T}$.  The set of $\Psi$-circuits of $\mathcal{T}$ is denoted $C( \mathcal{T}, \Psi)$. Such $\Psi$-circuits are often circuits of a matroid in the following sense.

In Corollary 6.6 of ~\cite{Games}, the authors show that the Axiom of Determinacy is equivalent to the statement that every set $\Psi$ of ends of every tree of finite matroids of overlap $1$ induces a matroid. Regardless of axiom system used, a set $\Psi$ of ends of a tree of finite matroids of overlap $1$ induces an independence system with its $\Psi$-circuits. Let us name these independence systems $\Psi$-finite tree systems. It may be tempting to think that the Axiom of Determinacy strictly gives us more matroids because of their result. However, it seems that we also lose some matroids when we assume the Axiom of Determinacy. We will use the classical finitary matroids we introduced earlier on page 8 in the introduction to show precisely what we mean.

\begin{theorem}
It is impossible to pick an axiom system such that all $\Psi$-finite tree systems are matroids and all classical finitary matroids are matroids.
\end{theorem}

\begin{proof}
Suppose all $\Psi$-finite tree systems are matroids. Then the result in ~\cite{Games} shows that we are assuming the Axiom of Determinacy. Thus, the Axiom of Choice is false. By the result of ~\cite{VectorChoice}, the Axiom of Choice is equivalent to the statement that every vector space has a basis. So we can deduce that there exists some vector space $V$ with no basis. Consider $M=(E(V), \mathcal{L} (V))$ where $E(V)$ is the set of vectors of $V$ and $\mathcal{L} (V)$ is the set of linearly independent subsets of $E(V)$. It is quite clear that $M$ satisfies the axioms of a classical finitary matroid that we introduced earlier. The empty set is linearly independent. Subsets of linearly independent sets are linearly independent. If $A$ and $B$ are linearly independent subsets of $E(V)$ with $|A| < |B| < \infty$, then there is some $b \in B \setminus A$ such that $A+b$ is linearly independent. A subset $S$ of $E(V)$ is linearly independent if and only if all finite subsets of $S$ are independent.

Now we show that $M$ violates our fourth independence axiom for matroids. Note that $\emptyset \in \mathcal{L}(V)$ and consider the set

\begin{equation*}
X:=\{ S \in \mathcal{L} \colon \emptyset \subset S \subset E(V) \}.
\end{equation*}
$X$ is the set of linearly independent subsets of $V$. $X$ has no maximal element since $V$ has no basis.
\end{proof}

It thus seems that the classes of matroids we have depend on the axiom system we are using. Our argument also shows that the statement "Every classical finitary matroid is a matroid in our sense" is equivalent to the Axiom of Choice.

Bowler and Carmesin's $\Psi$-finite tree systems motivate us to generalize their constructions using the same idea in our nearly-$P( \Psi)$ constructions. They consider a subset $\Psi$ of the ends of a tree $T$. It seems natural to consider partitions of $\Psi$.
\end{section}

  \chapter[%
      Short Title of 6th Ch.
   ]{%
     Thin Sum Matroids
   }%
   \label{ch:6thChapterLabel}
   \begin{section}{Nearly-Thin Families}
In the introduction, we defined what it means for a family of functions to be nearly thin. Recall that an indexed family of functions $\mathcal{F}_E$ from $A$ to $k$ is called $n$-nearly thin if there are at most $n$ elements $a \in A$ such that there are infinitely many $e \in E$ with $f(e)(a)$ nonzero. Following ~\cite{Cocircuit}, we define a matroid $M$ to be $\textit{tame}$ if the intersection of any circuit of $M$ with any cocircuit of $M$ is finite. We define a matroid to be $\textit{wild}$ if it is not tame. In ~\cite{ThinShum}, the authors show that the class of tame thin sums matroids over a field $k$ is closed under duality in their Theorem 1.4. We will now list some conjectures.

\begin{conjecture}
If $M$ is a thin sums matroid and $n$ is the minimal number such that $M$ is thinly representable by an $n$-nearly thin family then $M^*$ is an $n$-nearly finitary matroid with $n$ being minimal.
\end{conjecture}
\begin{conjecture}
If $M$ is a thin sums matroid that is not thinly representable by any nearly thin family, then $M^*$ is not nearly finitary.
\end{conjecture}
\begin{conjecture}
If $M$ is an $n$-nearly finitary tame thin sums matroid, then its dual tame thin sums matroid is thinly representable by an $n$-nearly thin family.
\end{conjecture}
Since every nearly thin family is $n$-nearly thin for some $n \in \mathbb{N}$, proving these conjectures could lead to a proof that nearly finitary tame thin sums matroids are $n$-nearly finitary for some $n$.

We now introduce an example of a thin sums matroid over a $1$-nearly thin family whose dual is $1$-nearly finitary. Let $k$ be any field and let $A={a}$ be a singleton set. Consider the family of functions $\mathcal{F}_{\mathbb{N}}$ defined by the rule $f(n)(a)=1$ for all $n \in \mathbb{N}$. Let $M = (\mathbb{N}, \mathcal{L}(M))$ where $\mathcal{L}(M)$ consists of thinly independent subsets of $\mathbb{N}$. Then $M$ is a matroid whose independent sets are sets of cardinality at most $1$. Independent sets in $M^*$  consists of subsets of $\mathbb{N}$ whose complement has cardinality at least $1$. Every finite subset of $\mathbb{N}$ is independent in $M^*$. So the finitarization of $M^*$  is $(\mathbb{N}, 2^{\mathbb{N}})$. $M^*$ is $1$-nearly finitary and $M$ is a thin sums matroid over a $1$-nearly thin family. This example is consistent with conjectures $6.1.1$ and $6.1.3$. Possible future directions include finding more substantial evidence for these conjectures.
\end{section}

\begin{section}{Topological Matroids}
Thin sums representability is an interesting generalization of representability in matroids. Here, we introduce a new notion of representability for matroids. Let $E$ be a countable set, $V$ be a topological vector space over a field $k$, and let $f: E \rightarrow V$. We define
\begin{equation*}
C(f):= \{ c \in k^E , \sum_{e \in E} c(e)f(e) \text{ converges in the topology of } V \}.
\end{equation*}

When we say an infinite countable sum converges in a topology, we mean that one of its sequences of partial sums converges absolutely.

We also define a function $\hat{f}: C(f) \rightarrow V$ by 
\begin{equation*}
\hat{f} (c) := \sum_{e \in E} c(e)f(e).
\end{equation*}

For $S \subset E$, we define 

\begin{equation*}
C_S (f):=\{ c \in C \colon c(e)=0 \text{ } \forall e \in E \setminus S \}.
\end{equation*}

We then define $\hat{f}|_S$ to be the restriction of $\hat{f}$ to $C_S (f)$.

By construction, the above sum converges on the domain of $\hat{f}$. We define

\begin{equation*}
\mathcal{L} (f) := \{ S \subset E \text{ s.t. } \ker ( \hat{f}|_S ) = \{ 0 \} \}.
\end{equation*}

This gives us an independence system $M_f = (E, \mathcal{L} (f))$. We will refer to systems constructed this way as $\textit{topological independence systems}$. Whenever we have a topological independence system that is also a matroid, we say that it is a $\textit{topological matroid}$. Let $V=k^A$ for some set $A$ and consider the following topology on $k^A$. We put the discrete topology on $k$ and then we give $k^A$ the topology of pointwise convergence. In this case, our construction becomes a thin sums system as defined on page 15 of this thesis.

Without lost of generality, we can assume that $V= \mathrm{Im}( \hat{f} )$. This motivates us to consider the following:

\begin{equation*}
\mathcal{L}^\mathrm{iso} (f) := \{ S \in \mathcal{L} (f) \text{ s.t. } \mathrm{Im} ( \hat{f}|_S ) = \mathrm{Im} ( \hat{f} ) \}.
\end{equation*}

We want to characterize maximal independent sets. In particular, we guess that $\mathcal{L}^\mathrm{iso} (f) = \mathcal{L}^\mathrm{max} (f)$. Unfortunately, this will not always be the case. To see why, consider the Bean graph $G$ as introduced on page \pageref{BeanGraph} of the introduction to this thesis. For the reader's convenience, we include a picture of the Bean graph in Figure 6.1.

\begin{figure}
\includegraphics{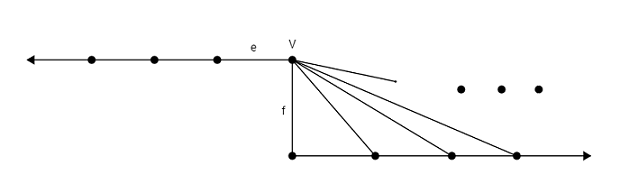}
\caption{Bean Graph}
\label{BeanGraph3}
\end{figure}

Let $V(G)$ be the set of vertices of this Bean graph. Let $k=\mathbb{F}_2$ and consider $k^{V(G)}$ as a vector space. Follow the construction of $M_{ACthin}(G)$ on page 14-15 of our introduction to define $f$. From this $f$ applied to our above construction with the discrete topology on $k$ and the topology of pointwise convergece on $k^{V(G)}$, we get a topological independence system that is also a thin sums system. By the proof on pages 14-15 of our introduction, this is also the algebraic cycle system of the Bean graph. Because the labelled vertex $v$ has infinitely many adjacent edges, $C(f)$ is the set of elements $c$ of $k^E$ where $c(e)$ has finite support on the set of edges adjacent to $v$. Consider the set $S$ of all top edges and all edges adjacent to vertex $v$ labelled in the graph. This $S$ is a maximal independent set since it is an algebraic spanning tree. However, $S$ is not an element of $\mathcal{L}^\mathrm{iso} (f)$ since $S$  $\mathrm{Im} ( \hat{f}|_S )$ cannot contain any element of $k^{V(G)}$ whose support has infinitely many lower vertices. Using an appropriate subset of the lower edges, it is possible to have an element in $\mathrm{Im} ( \hat{f} ) \subset k^{V(G)}$ whose support has infinitely many lower vertices. Thus, $\mathrm{Im} ( \hat{f}|_S )$ cannot be all of $\mathrm{Im} ( \hat{f})$.

The following theorem gives us a condition for when an indepedent subset of $M_f$ is a base. 
\begin{theorem}
Let $M_f$ be a topological independence system. An independent set $S$ of $M_f$ is maximal if and only if for all $e\in E$, 
\begin{equation*}
f(e)\in\hat{f}\left(C_{S}(f)\right).
\end{equation*}
\end{theorem}

\begin{proof}
Suppose $S$ is independent and there is some $e$ such that
\begin{equation*}
f(e)\notin\hat{f}\left(C_{S}(f)\right).
\end{equation*}

 This implicitly assumes that $e\notin S$. Lets assume that $S+e=S\cup\{e\}$ is dependent. Then any nontrivial dependence $c$ of $S+e$ must have $c(e)\neq0$. Then

\begin{equation*}
-c(e)f(e)=\hat{f}(c)-c(e)f(e).
\end{equation*}

 Notice, $\hat{f}(c)-c(e)f(e)$ can be written as an element of $\hat{f}\left(C_{S}(f)\right)$ since the coefficient of $f(e)$ is $c(e)-c(e)=0$. We conclude that $f(e)\in\hat{f}\left(C_{S}(f)\right)$. Thus, this contradiction shows that $S+e$ must be an independent set as well. So $S$ is not maximal.

Conversely, suppose that $S$ is independent and $f(e)\in\hat{f}\left(C_{S}(f)\right)$ for all $e\in E$. If $S=E$, $S$ is obviously maximal. Otherwise, assume there is some $e_{0}\in E\setminus S$. Then there is some $c\in C_{S}(f)$ such that 
\begin{equation*}
\hat{f}\left(C_{S}(f)\right)=f(e_{0}).
\end{equation*}
  We define a function $c_{0}:E\rightarrow k$ as follows. We set $c_{0}(s)=c(s)$ for all $s\in S$, $c_{0}(e_{0})=-1$ and $c_{0}(e)=0$ for all $e\in E\setminus\left(S\cup\{e_{0}\}\right)$. This $c_{0}$ is a nontrivial dependence for $S\cup\{e_{0}\}$. Thus, adding any new element to $S$ will give us a dependent set. So $S$ is maximal.
\end{proof}
\end{section}
   

       
   \backmatter
   
   \bibliographystyle{amsalpha-fi-arxlast}
   \bibliography{DissertationBibliography}

\newcommand{\etalchar}[1]{$^{#1}$}
\providecommand{\bysame}{\leavevmode\hbox to3em{\hrulefill}\thinspace}
\providecommand{\MR}{\relax\ifhmode\unskip\space\fi MR }
\providecommand{\MRhref}[2]{%
  \href{http://www.ams.org/mathscinet-getitem?mr=#1}{#2}
}
\providecommand{\href}[2]{#2}
\begin{thebibliography}{BDK{\etalchar{+}}13}

\bibitem[AB15]{ThinShum}
H.~Afzali and N.~Bowler, \emph{Thin sums matroids and duality}, Advances in
  Mathematics \textbf{271} (2015), 1--29, \mbox{arXiv:1204.6294}.

\bibitem[AHCF11]{Union}
E.~Aigner-Horev, J.~Carmesin, and J.-O. Fr{\"o}hlich, \emph{Infinite matroid
  union}, 2011, \mbox{arXiv:1111.0602}.

\bibitem[AHCF12]{Inter}
\bysame, \emph{On the intersection of infinite matroids}, 2012,
  \mbox{arXiv:1111.0606}.

\bibitem[BC13]{Games}
N.~Bowler and J.~Carmesin, \emph{Infinite {M}atroids and {D}eterminacy of
  {G}ames}, 2013, \mbox{arXiv:1301.5980}.

\bibitem[BC14]{Cocircuit}
\bysame, \emph{Matroids with an infinite circuit-cocircuit intersection}, J.
  Comb. Theory, Ser. {B} \textbf{107} (2014), no.~0, 78--91,
  \mbox{arxiv:1202.3406}.

\bibitem[BCC13]{InfGraphic}
N.~Bowler, J.~Carmesin, and R.~Christian, \emph{Infinite graphic matroids
  {P}art {I}}, 2013, \mbox{arXiv:1309.3735}.

\bibitem[BCP15]{Dray}
N.~Bowler, J.~Carmesin, and J.~Pott, \emph{Edge-disjoint double rays in
  infinite graphs: a {H}alin type result}, J. Comb. Theory, Ser. {B}
  \textbf{111} (2015), 1--16, \mbox{arXiv:1307.0992}.

\bibitem[BD11]{InfGraph}
H.~Bruhn and R.~Diestel, \emph{Infinite matroids in graphs}, Discrete
  Mathematics \textbf{311} (2011), 1461--1471, \mbox{arXiv:1011.4749}.

\bibitem[BDK{\etalchar{+}}13]{Axm}
H.~Bruhn, R.~Diestel, M.~Kriesell, R.~Pendavingh, and P.~Wollan, \emph{Axioms
  for infinite matroids}, Advances in Mathematics \textbf{239} (2013), 18--46,
  \mbox{arXiv:1003.3919}.

\bibitem[Bla84]{VectorChoice}
A.~Blass, \emph{Bases in vector spaces and the axiom of choice}, Contemporary
  Mathematics \textbf{31} (1984), 31--33.

\bibitem[Bow]{Blog}
N.~Bowler, \emph{Infinite {T}rees of {M}atroids},
  \url{http://matroidunion.org/?p=1809}, Accessed: 2017-02-04.

\bibitem[Car17]{Top}
J.~Carmesin, \emph{Topological cycle matroids of infinite graphs}, European
  Journal of Combinatorics \textbf{60} (2017), 135--150,
  \mbox{arXiv:1412.0830}.

\bibitem[CT08]{graphlike}
A.~V. Carsten~Thomassen, \emph{Graph-like continua, augmenting arcs, and
  {M}enger’s theorem}, Combinatorica \textbf{28} (2008), no.~595,
  doi:10.1007/s00493-008-2342-9.

\bibitem[Die11]{Diestel}
R.~Diestel, \emph{Locally finite graphs with ends: a topological approach. {I}.
  {B}asic theory}, Discrete Mathematics \textbf{311} (2011), 1423--1447,
  \mbox{arXiv:0912.4213}.

\bibitem[Elm16]{Cofinitary}
A.-K. Elm, \emph{{Investigations in infinite matroid theory}}, Master's thesis,
  University of Hamburg, 2016.

\bibitem[Hal65]{Halin}
R.~Halin, \emph{Über die {M}aximalzahl fremder unendlicher {W}ege in
  {G}raphen}, Mathematische Nachrichten \textbf{30} (1965), 63--85,
  doi:10.1002/mana.19650300106.

\bibitem[Hal70]{Halindray}
\bysame, \emph{Die {M}aximalzahl fremder zweiseitig unendlicher {W}ege in
  {G}raphen}, Mathematische Nachrichten \textbf{44} (1970), 119--127.

\bibitem[Hig69a]{Higgs}
D.~A. Higgs, \emph{Infinite graphs and matroids}, Recent Progress in
  Combinatorics, Proceedings Third Waterloo Conference on Combinatorics,
  Academic Press (1969), 245--253.

\bibitem[Hig69b]{Bmatroid}
\bysame, \emph{Matroids and {D}uality}, Colloq. Math. \textbf{20} (1969),
  215--220.

\bibitem[Oxl78a]{OxleyI}
J.~Oxley, \emph{Infinite matroids}, Proc. London Math. Soc. \textbf{37} (1978),
  no.~3, 259--272.

\bibitem[Oxl78b]{OxleyII}
\bysame, \emph{Some problems in combinatorial geometries}, Ph.D. thesis, Oxford
  University, 1978.

\bibitem[Oxl92]{OxleyIII}
\bysame, \emph{Matroid theory}, Oxford University Press (1992).

\bibitem[SHAB16]{Gammoids}
M.~M. Seyed Hadi Afzali~Borujeni, Hiu-Fai~Law, \emph{Finitary and {C}ofinitary
  {G}ammoids}, Discrete Applied Mathematics \textbf{209} (2016), 2--10.

\bibitem[Wel10]{Welsh}
D.~J.~A. Welsh, \emph{Matroid theory}, Dover Publications, 2010.

\end{thebibliography}
\end{document}